\documentclass[11pt,oneside,reqno]{amsart}

\usepackage{graphicx}
\usepackage{pict2e}
\usepackage{amssymb}
\usepackage{amsthm}                % better theorem environments
\usepackage[margin=1in]{geometry}  % set the margins to 1in on all sides
\usepackage{graphics}
\usepackage{epstopdf}
\usepackage{amsmath}
\usepackage{graphicx}
\usepackage{subfigure}
\usepackage{latexsym}
\usepackage{amsmath}
\usepackage{amsfonts}
\usepackage{amssymb}
\usepackage{tikz}
\usepackage{mathdots}
\usepackage{float}
\usepackage[mathscr]{euscript}
\usepackage{enumerate}
\usepackage{epsfig}

\theoremstyle{definition}
\newtheorem{theorem}{Theorem}[section]
\newtheorem{lemma}[theorem]{Lemma}
\newtheorem{corollary}[theorem]{Corollary}

\theoremstyle{definition}
\newtheorem{definition}[theorem]{Definition}

\theoremstyle{remark}

\numberwithin{equation}{section}

\numberwithin{equation}{section}
\begin{document}

\title{The Colored Kauffman Skein relation and the head and tail of the colored Jones polynomial}

%%\author{Oliver Dasbach}
%%\address{Department of Mathematics, Louisiana State University, 
%%Baton Rouge, LA 70803 USA}
%%\email{kasten@math.lsu.edu}
%%\urladdr{www.math.sc.edu/$\sim$howard} % Delete if not wanted.

%%
%% If there is another author uncomment and edit the following.
%%

\author{Mustafa Hajij}
\address{Department of Mathematics, Louisiana State University, 
Baton Rouge, LA 70803 USA}
\email{mhajij1@math.lsu.edu}

%%
%% If there are three of more authors they are added in the obvious
%% way. 
%%

%%%
%%% The following is for the abstract.  The abstract is optional and
%%% if not used just delete, or comment out, the following.
%%%

\begin{abstract}
Using the colored Kauffman skein relation, we study the highest and lowest $4n$ coefficients of the $n^{th}$ unreduced colored Jones polynomial of alternating links. This gives a natural extension of a result by Kauffman in regard with the Jones polynomial of alternating links and its highest and lowest coefficients. We also use our techniques to give a new and natural proof for the existence of the tail of the colored Jones polynomial for alternating links.   
\end{abstract}

 \maketitle

\section{Introduction}
The unreduced colored Jones polynomial is link invariant that assigns to each link $L$ a sequence of Laurent polynomials $\tilde{J}_{n,L}$ indexed by a positive integer $n$, the color. In \cite{Dasbach} Dasbach and Lin conjectured that, up to a common sign change, the highest $4(n+1)$ (the lowest resp.) coefficients of the polynomial $\tilde{J}_{n,L}(A)$ of an alternating link $L$ agree with the first $4(n+1)$ coefficient of the polynomial $\tilde{J}_{n+1,L}(A)$ for all $n$. This gives rise to two power series with integer coefficients associated with the alternating link $L$ called the head and the tail of the colored Jones polynomial. The existence of the head and tail of the colored Jones polynomial of adequate links, a class of links that contains alternating links, was proven by Armond in \cite{Cody2} using skein theory. Independently, this was shown by Garoufalidis and Le for alternating links using $R$-matrices \cite{klb} and generalized to higher order tails.\\

Write $\mathcal{S}(S^{3})$ to denote the Kauffman Bracket Skein Module of $S^{3}$. Let $L$ be an alternating link and let $D$ be a reduced link diagram of $L$. Write $S_A(D)$ to denote the $A$-smoothing state of $D$, the state obtained by replacing each crossing by an $A$-smoothing. The state $S_B(D)$ of $D$ is defined similarly. Using the Kauffman skein relation Kauffman \cite{Kauffman0} showed that the $A$ state (respectively the $B$ state) realizes the highest (respectively the lowest) coefficient of the Jones polynomial of an alternating link. We extend this result to the colored Jones polynomial by using \textit{the colored Kauffman skein relation} \ref{section3}. We show that the $n$-colored $A$ state and the $n$-colored $B$ state realize the highest and the lowest $4n$ coefficients of the $n^{th}$ unreduced colored Jones polynomial of an alternating link. Furthermore we show that this gives a natural layout to prove the stability of the highest and lowest coefficients of the colored Jones polynomial of alternating links. In other words we prove that the head and the tail of the colored Jones polynomial exist.
\section{background}
In this section we recall the definition of adequate links and its connection to the Kauffman states. We also review the basic of the Kauffman bracket skein module and the Jones-Wenzl idempotent.
\subsection{Alternating Links}
Let $L$ be a link in $S^3$ and $D$ be an alternating knot diagram of $L$. For any crossing in $D$ there are two ways to smooth this crossing, the $A$-smoothing and the $B$-smoothing. See Figure \ref{smoothings1}.
\begin{figure}[H]
  \centering
   {\includegraphics[scale=0.14]{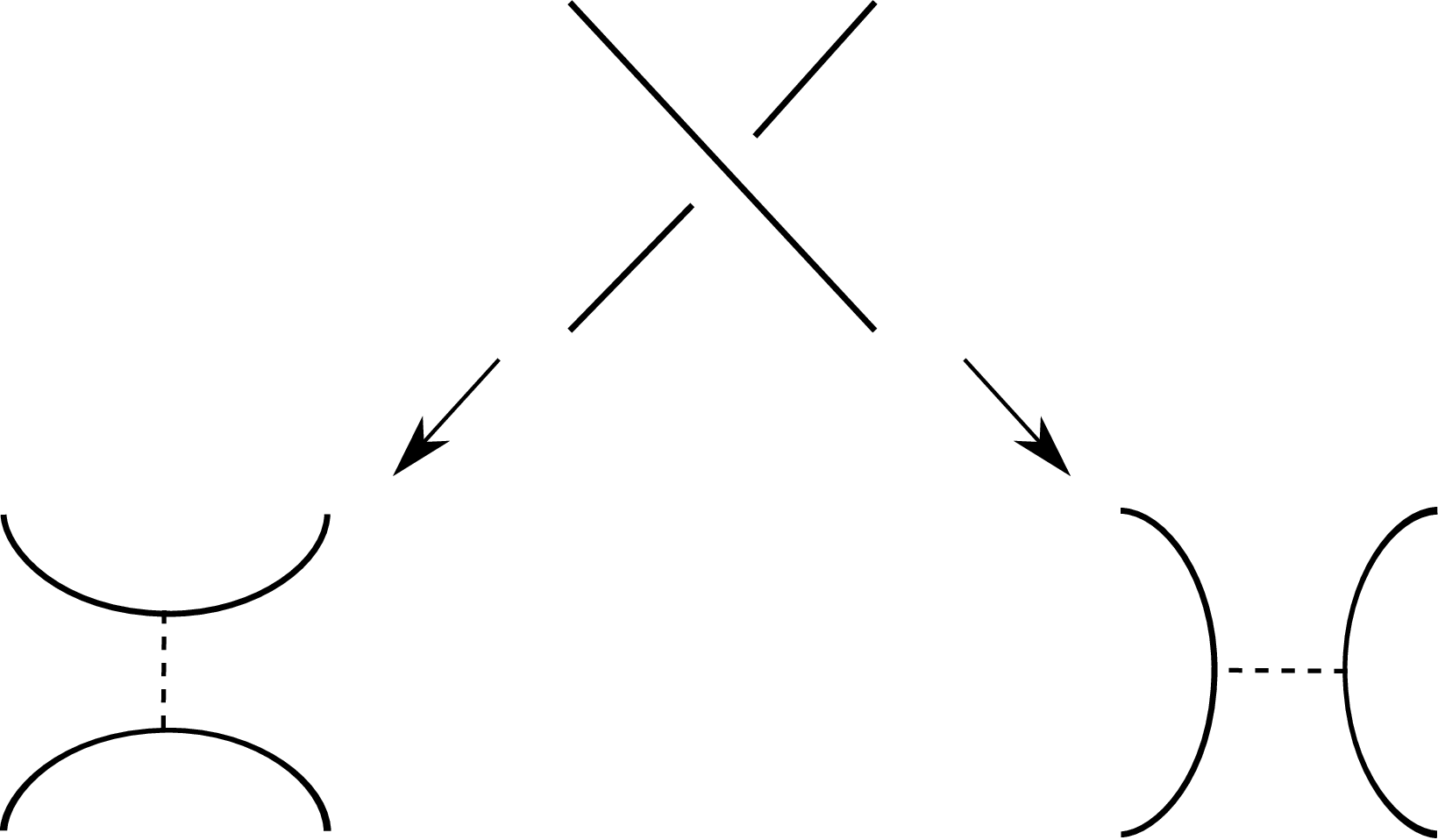}
    \put(-88,33){$A$}
          \put(-30,33){$B$}
    \caption{A and B smoothings}
  \label{smoothings1}}
\end{figure}

We replace a crossing with a smoothing together with a dashed line joining the two arcs. After applying a smoothing to each crossing in $D$ we obtain a planar diagram consisting of collection of disjoint circles in the plane. We call this diagram a \textit{state} for the diagram $D$. The $A$-smoothing state, obtained from $D$ by replacing every crossing by an $A$ smoothing, and the $B$-smoothing state for $D$ are of particular importance for us. Write $S_A(D)$ and $S_B(D)$ to denote the $A$ smoothing and $B$ smoothing states of $D$ respectively. For each state $S$ of a link diagram $D$ one can associate a graph $G_{S(D)}$ obtained by replacing each circle of $S$ by a vertex and each dashed link by an edges. 
\begin{definition}
A link diagram $D$ is called $A$-adequate ($B$-adequate, respectively) if there are no loops in the graph $G_{S_A(D)}$ (the graph $G_{S_A(D)}$, respectively). A link diagram $D$ is called adequate if it is both $A$-adequate and $B$-adequate.
\end{definition}
It is known that a reduced alternating link diagram is adequate. See for example \cite{Lickorish1}. It is important to see that the if a link diagram is $A$-adequate then any state that has only one $B$-smoothing will have one fewer circle than the $A$-smoothing state.
\subsection{Skein Theory}
Let $M$ denotes a $3$-manifold that is homeomorphic to $F\times [0,1]$, where $F$ is a connected oriented surface. We will also use $ \mathbb{Q}(A)$ to denote the field generated by the indeterminate $A$ over the rational numbers.
\begin{definition}(J. Przytycki \cite{Przytycki} and V. Turaev \cite{tur})
 The \textit{Kauffman Bracket Skein Module} of $M=F\times I$, denoted by $\mathcal{S}(F)$, is the $ \mathbb{Q}(A)$-free module generated by the set of isotopy classes of framed links in $M$(including the empty knot) modulo the submodule generated by the Kauffman relations \cite{Kauffman1}:
\begin{eqnarray*}(1)\hspace{3 mm}
  \begin{minipage}[h]{0.06\linewidth}
        \vspace{0pt}
        \scalebox{0.04}{\includegraphics{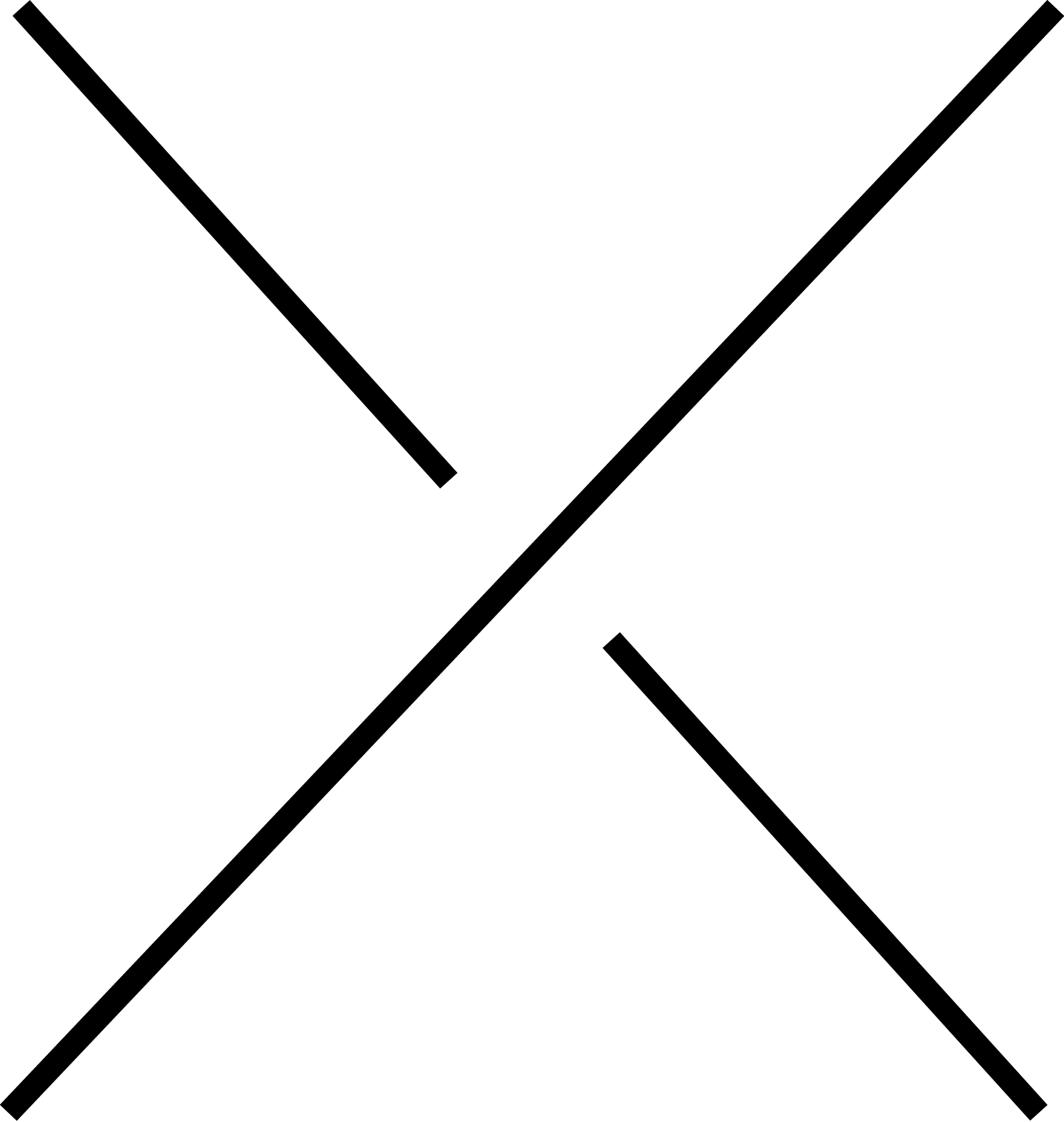}}
   \end{minipage}
   -
     A 
  \begin{minipage}[h]{0.06\linewidth}
        \vspace{0pt}
        \scalebox{0.04}{\includegraphics{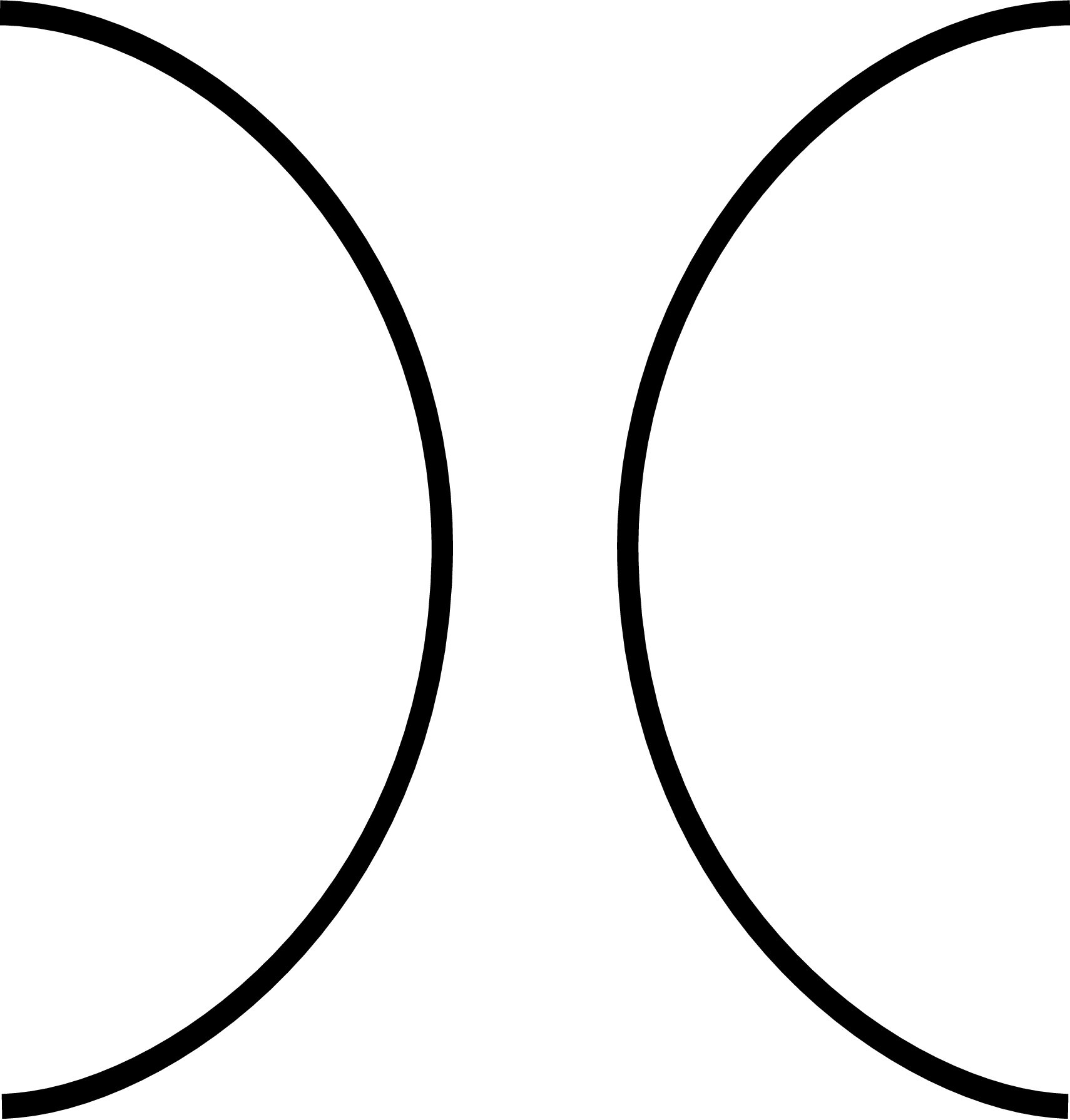}}
   \end{minipage}
   -
  A^{-1} 
  \begin{minipage}[h]{0.06\linewidth}
        \vspace{0pt}
        \scalebox{0.04}{\includegraphics{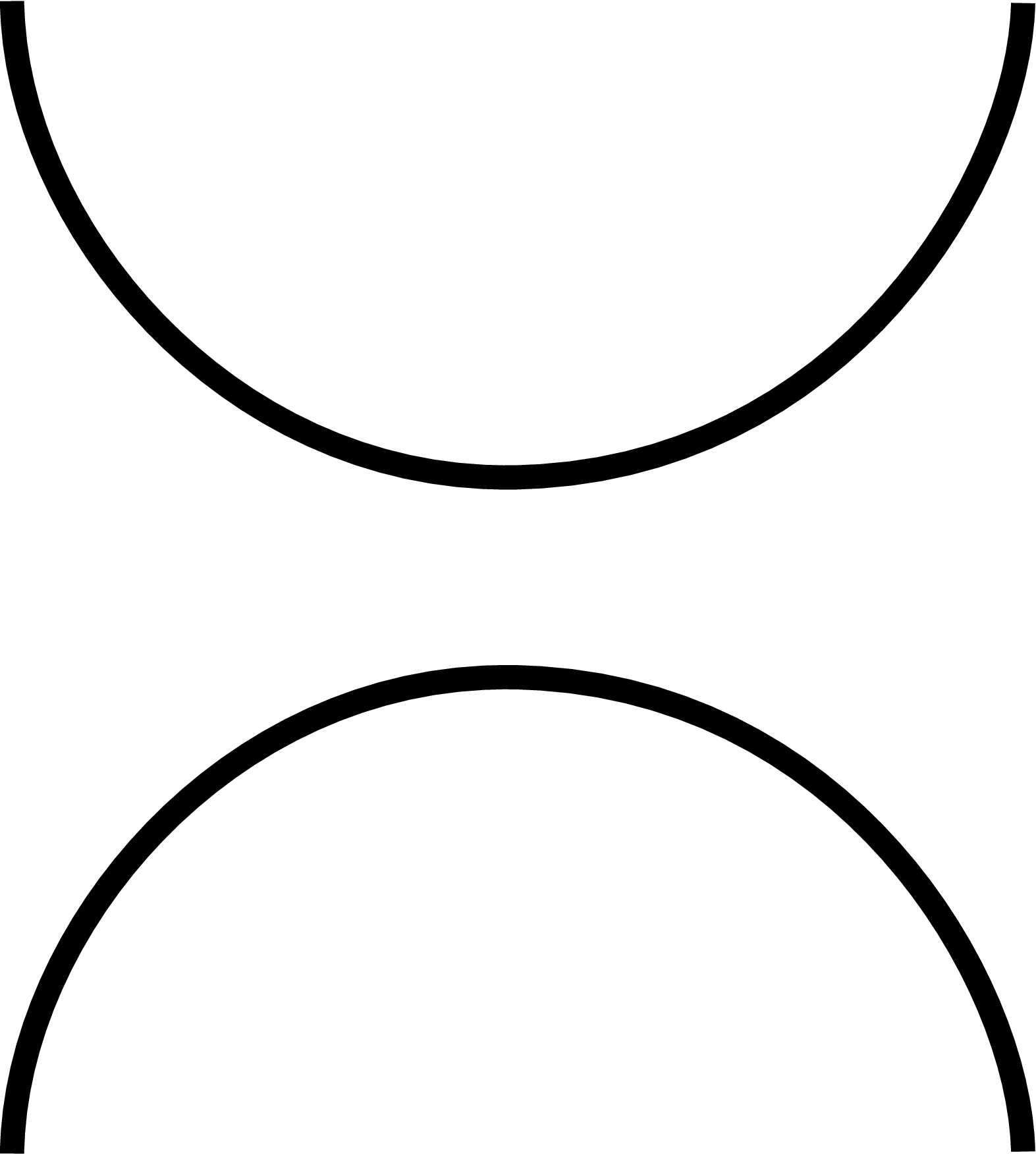}}
   \end{minipage}
, \hspace{20 mm}
  (2)\hspace{3 mm} L\sqcup
   \begin{minipage}[h]{0.05\linewidth}
        \vspace{0pt}
        \scalebox{0.02}{\includegraphics{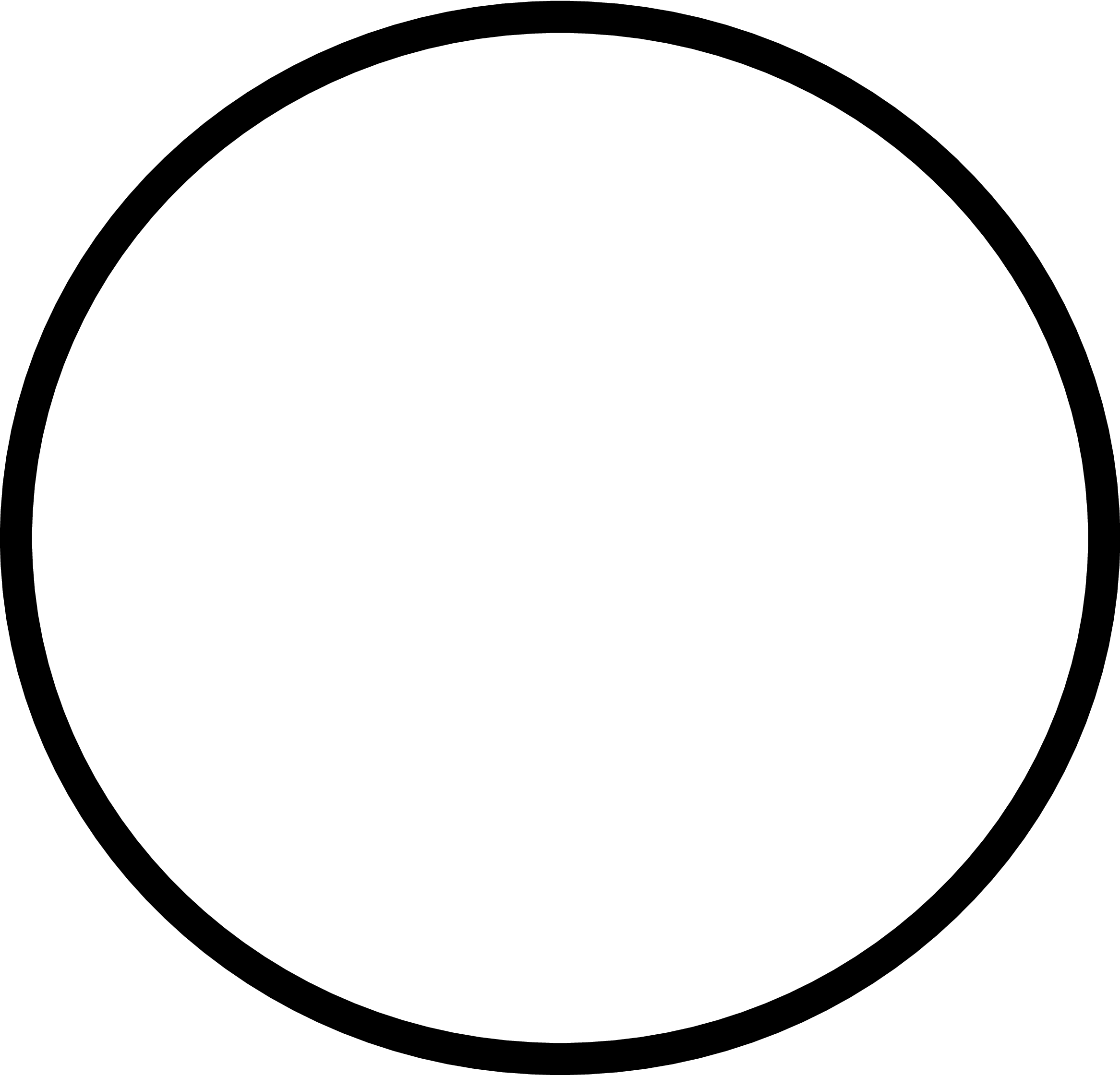}}
   \end{minipage}
  +
  (A^{2}+A^{-2})L. 
  \end{eqnarray*}
where $L\sqcup$ \begin{minipage}[h]{0.05\linewidth}
        \vspace{0pt}
        \scalebox{0.02}{\includegraphics{simple-circle}}
   \end{minipage}  consists of a framed link $L$ in $M$ and the trivial framed knot 
   \begin{minipage}[h]{0.05\linewidth}
        \vspace{0pt}
        \scalebox{0.02}{\includegraphics{simple-circle}}
   \end{minipage}.
 \end{definition}    
Relation (1) is usually called \textit{the Kauffman skein relation}. Since the manifold $M$ is homeomorphic to $F\times [0,1]$ one may consider an appropriate version of link diagrams in $F$ to represent the classes of $\mathcal{S}(F)$ instead of framed links in $M$. If the manifold $M$ has a boundary then one could define a relative version of the Kauffman bracket skein module.

Our first example is the Kauffman bracket skein module of the sphere $\mathcal{S}(S^{2})$. This module can be easily seen to be isomorphic to $ \mathbb{Q}(A)$. This isomorphism is provided by the Kauffman bracket  and it is induced by sending $D$ to $\langle D \rangle$. In particular it send the empty link to $1$. Along with the module $\mathcal{S}(S^{2})$, we will also need the relative skein module $\mathcal{S}(D^2,2n)$, where the rectangular disk $D^2$ has $n$ designated points on the top edge and $n$ designated points on the bottom edge. In fact, the module $\mathcal{S}(D^2,2n)$ admits a multiplication given by vertical concatenation of two diagrams in $\mathcal{S}(D^2,2n)$. With this multiplication $\mathcal{S}(D^2,2n)$ becomes an associative algebra over $Q(A)$ known as the \textit{$n^{th}$ Temperley-Lieb algebra} $TL_n$.\\

For each positive integer $n$, the algebra $TL_n$ contains an element of special importance to us.  The$n^{th}$ \textit{Jones-Wenzl idempotent} (projector), denoted $f^{(n)}$, was first discovered by Jones \cite{Jones}. We shall adapt a diagrammatic presentation for Jones-Wenzel projector that is due Lickorish \cite{Lickorish3}. In this notation one thinks of $f^{(n)}$ as an empty box with $n$ strands entering and $n$ strands leaving the opposite side. See Figure \ref{JW}.
\begin{figure}[H]
  \centering
   {\includegraphics[scale=0.13]{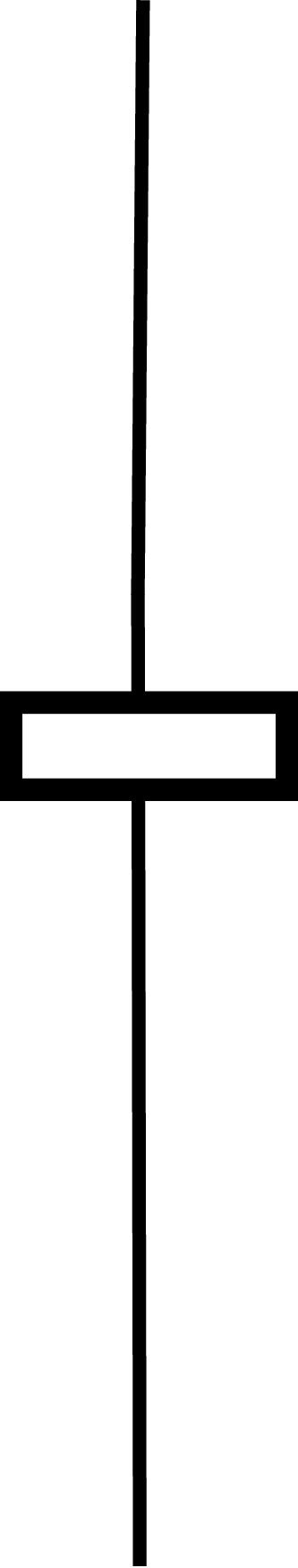}
   \put(-20,+70){\footnotesize{$n$}}
    \caption{The $n^{th}$ Jones-Wenzl idempotent}
  \label{JW}}
\end{figure}
The Jones-Wenzl idempotent has two defining properties. The defining properties of $f^{(n)}$ are:
\begin{align}
  (1)\hspace{-15 mm}\label{properties8}
\hspace{0 mm}
  \begin{minipage}[h]{0.21\linewidth}
        \vspace{0pt}
        \scalebox{0.115}{\includegraphics{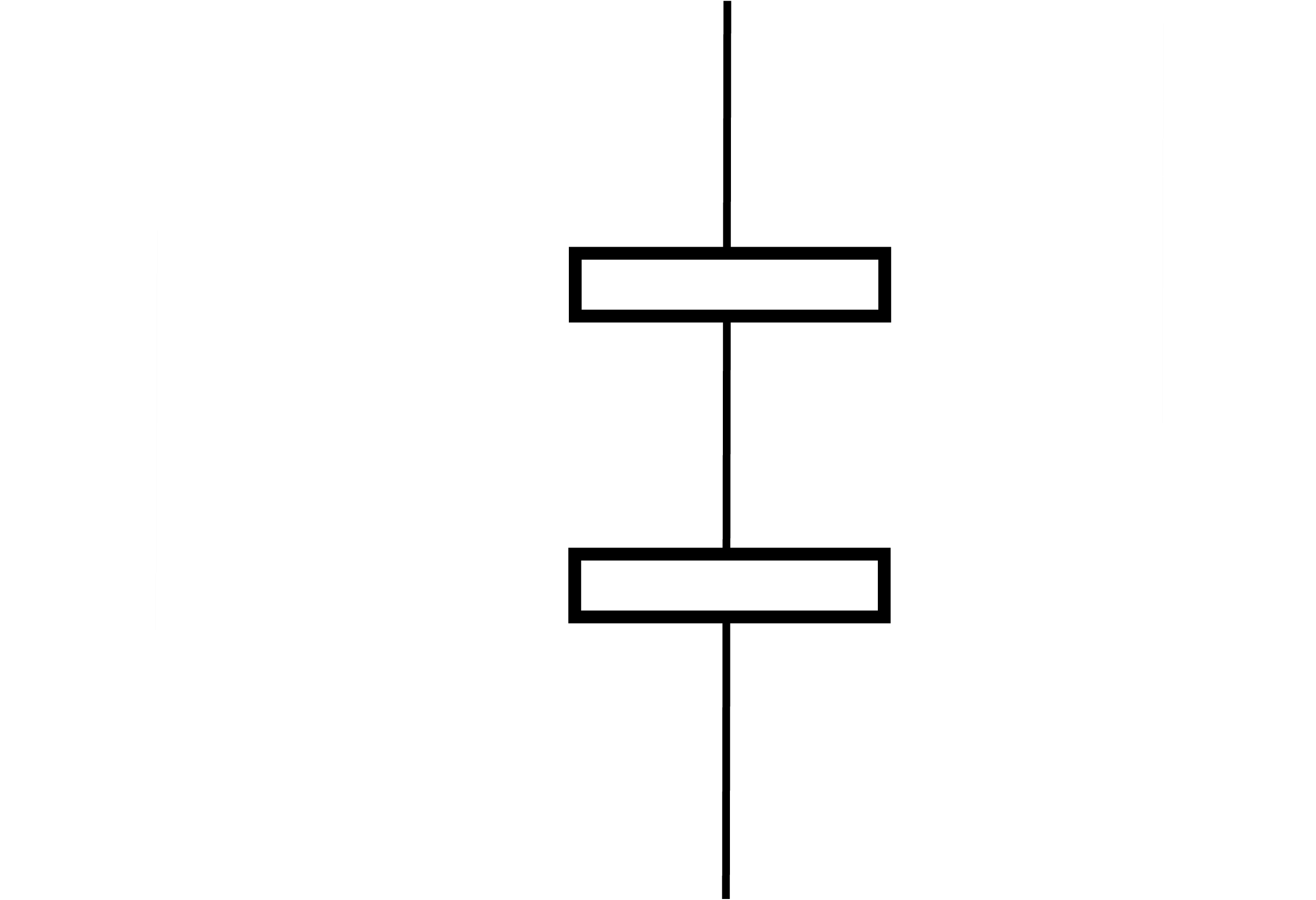}}
        \put(0,+80){\footnotesize{$n$}}    
   \end{minipage}
  = \hspace{5pt}
     \begin{minipage}[h]{0.1\linewidth}
        \vspace{0pt}
         \hspace{50pt}
        \scalebox{0.115}{\includegraphics{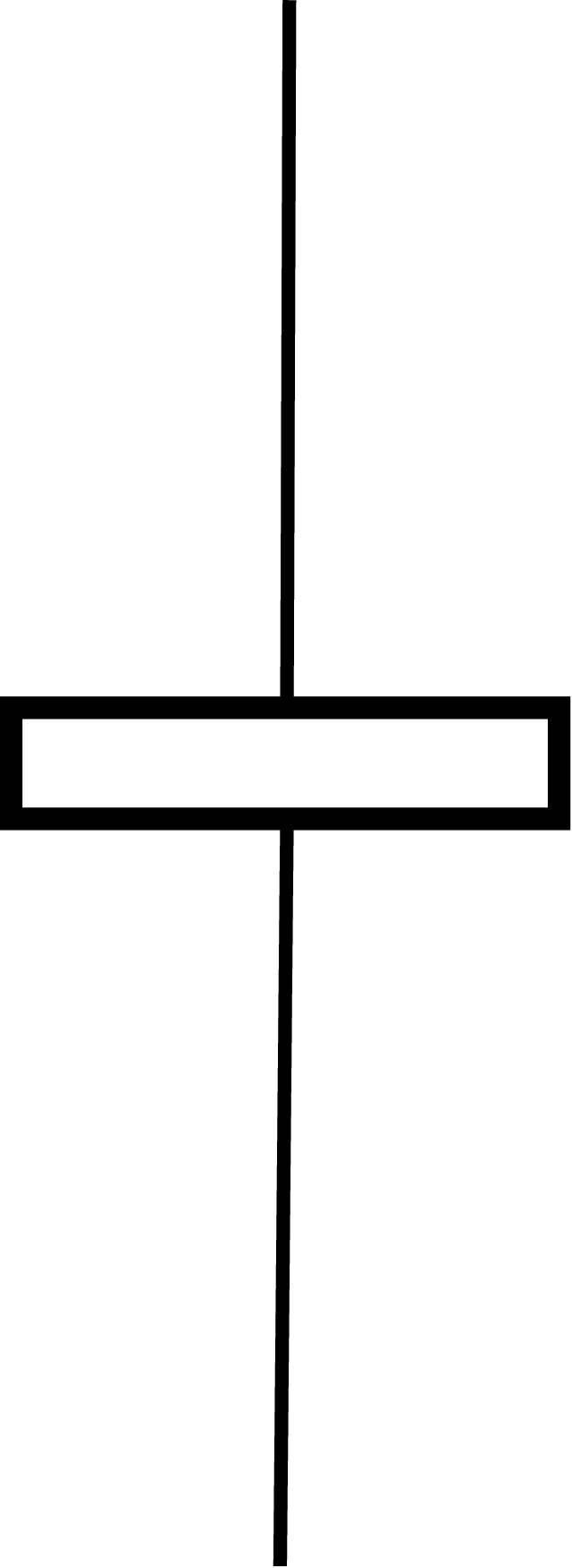}}
        \put(-60,80){\footnotesize{$n$}}
   \end{minipage}
    , (2) \hspace{22 mm}
    \begin{minipage}[h]{0.09\linewidth}
        \vspace{0pt}
        \scalebox{0.115}{\includegraphics{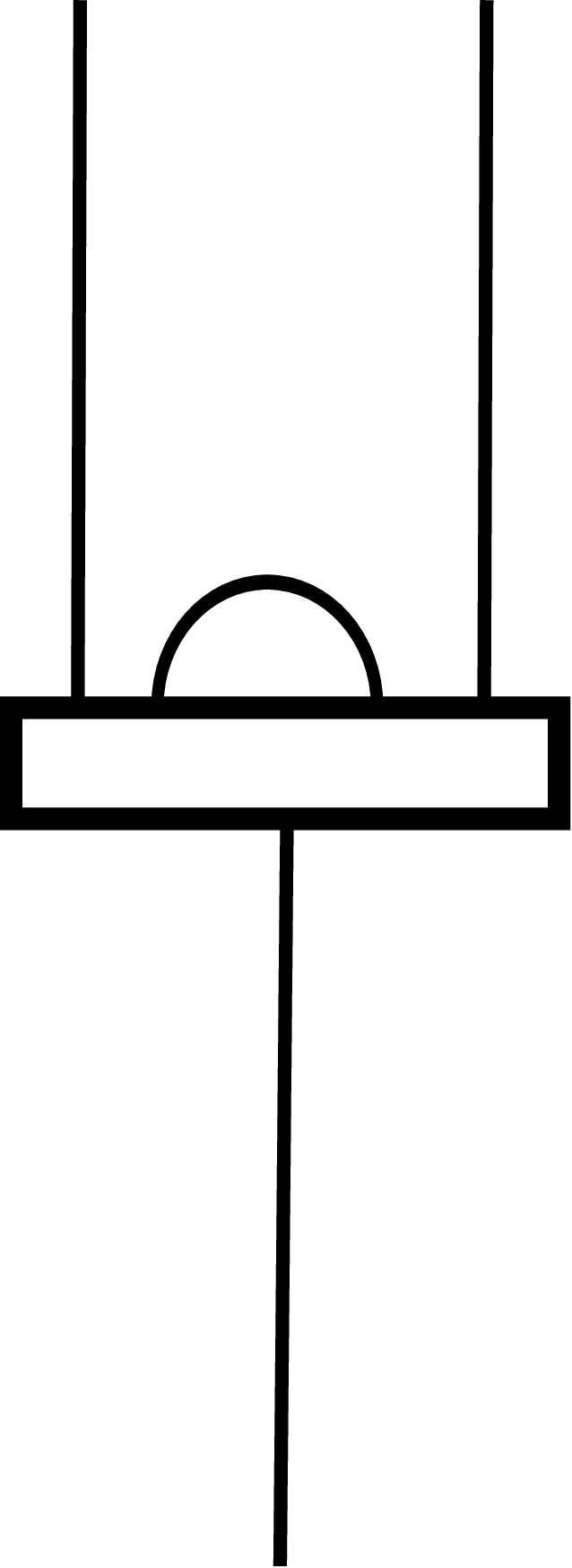}}
         \put(-70,+82){\footnotesize{$n-i-2$}}
         \put(-20,+64){\footnotesize{$1$}}
        \put(-2,+82){\footnotesize{$i$}}
        \put(-28,20){\footnotesize{$n$}}
   \end{minipage}
   =0.
     \end{align}
The second equation of \ref{properties8} holds for $1\leq i\leq n-1$ and it is usually called \textit{the annihilation axiom}. The $n^{th}$ Jones-Wenzl idempotent has an important recursive formula due to Wenzl \cite{Wenzl} and this formula is stated graphically as follows:   
\begin{align}
(1)\hspace{8 mm}\label{recursive}
  \begin{minipage}[h]{0.05\linewidth}
        \vspace{0pt}
        \scalebox{0.12}{\includegraphics{nth-jones-wenzl-projector}}
         \put(-20,+70){\footnotesize{$n$}}
   \end{minipage}
   =
  \begin{minipage}[h]{0.08\linewidth}
        \hspace{8pt}
        \scalebox{0.12}{\includegraphics{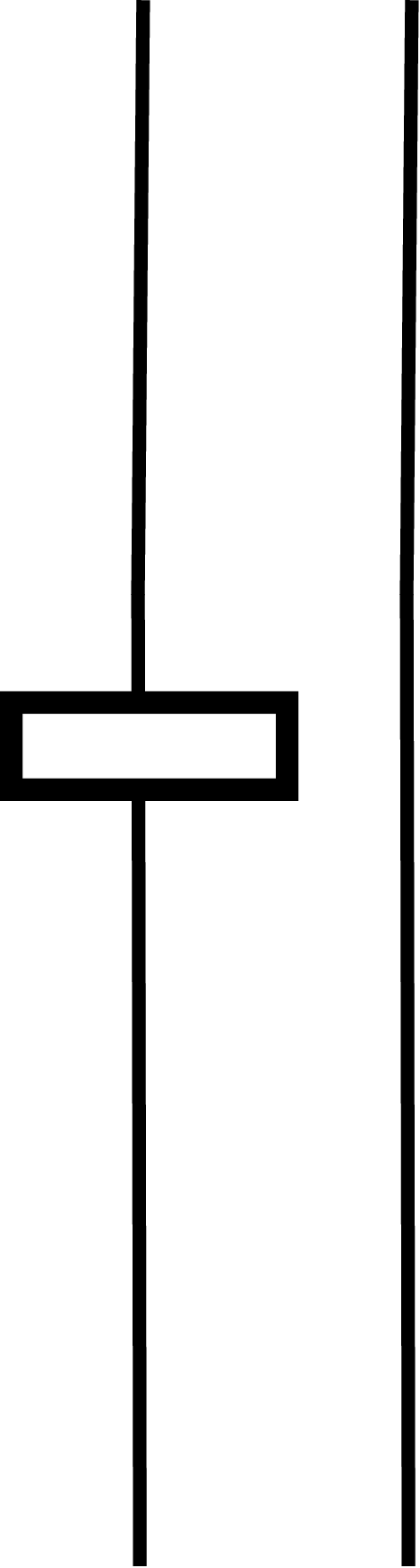}}
        \put(-42,+70){\footnotesize{$n-1$}}
        \put(-8,+70){\footnotesize{$1$}}
   \end{minipage}
   \hspace{9pt}
   -
 \Big( \frac{\Delta_{n-2}}{\Delta_{n-1}}\Big)
  \hspace{9pt}
  \begin{minipage}[h]{0.10\linewidth}
        \vspace{0pt}
        \scalebox{0.12}{\includegraphics{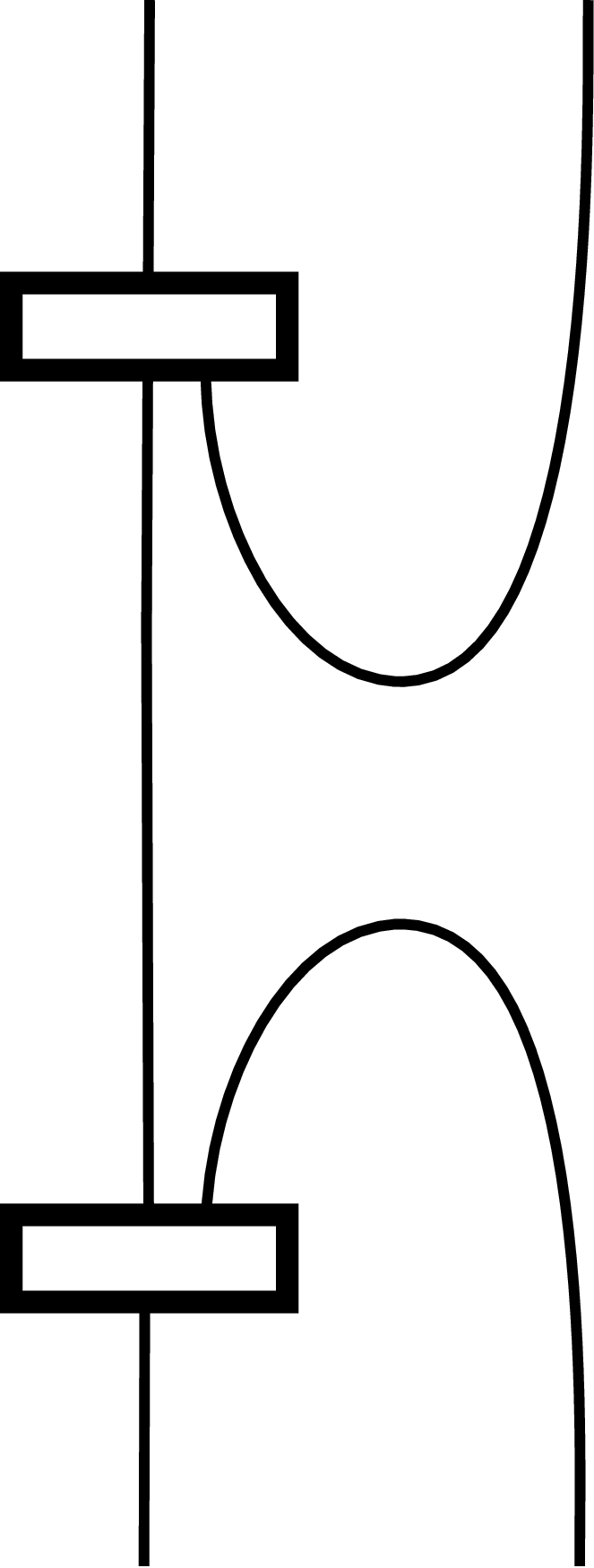}}
         \put(2,+85){\footnotesize{$1$}}
         \put(-52,+87){\footnotesize{$n-1$}}
         \put(-25,+47){\footnotesize{$n-2$}}
         \put(2,+10){\footnotesize{$1$}}
         \put(-52,+5){\footnotesize{$n-1$}}
   \end{minipage}
  , (2) \hspace{8 mm}
   \hspace{2 mm} \begin{minipage}[h]{0.05\linewidth}
        \vspace{0pt}
        \scalebox{0.12}{\includegraphics{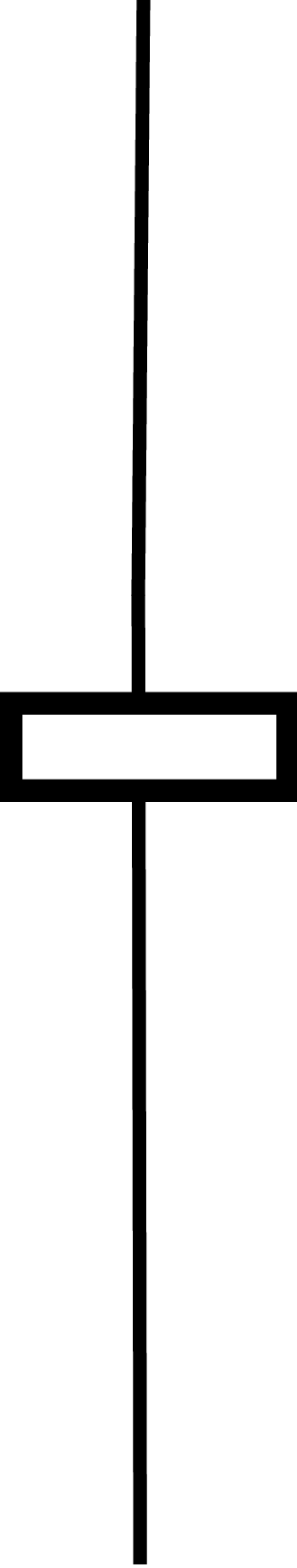}}
        \put(-20,+70){\footnotesize{$1$}}
   \end{minipage}
  =
  \begin{minipage}[h]{0.05\linewidth}
        \vspace{0pt}
        \scalebox{0.12}{\includegraphics{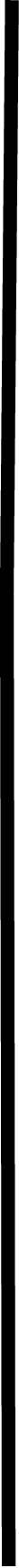}}
   \end{minipage}   
  \end{align}
  where 
\begin{equation*}
 \Delta_{n}=(-1)^{n}\frac{A^{2(n+1)}-A^{-2(n+1)}}{A^{2}-A^{-2}}.
\end{equation*}
We assume that $f^{(0)}$ is the empty tangle. A proof of Wenzl's formula can be found in \cite{Lickorish1} and \cite{Wenzl}. The defining properties of the Jones-Wenzl idempotent imply the following identities
\begin{eqnarray}
\label{properties}
\hspace{0 mm}
(1)\hspace{8 mm} \Delta_{n}=
  \begin{minipage}[h]{0.1\linewidth}
        \vspace{0pt}
        \scalebox{0.12}{\includegraphics{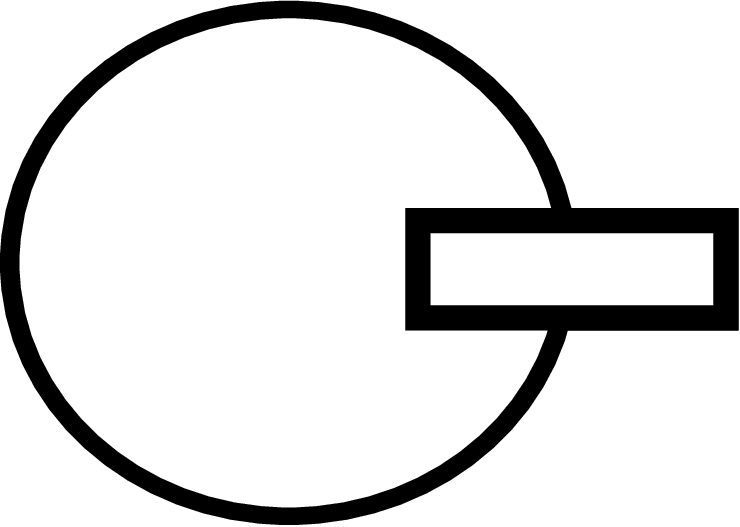}}
        \put(-29,+34){\footnotesize{$n$}}
   \end{minipage}
   , (2) \hspace{8 mm}
     \begin{minipage}[h]{0.08\linewidth}
        \vspace{0pt}
        \scalebox{0.115}{\includegraphics{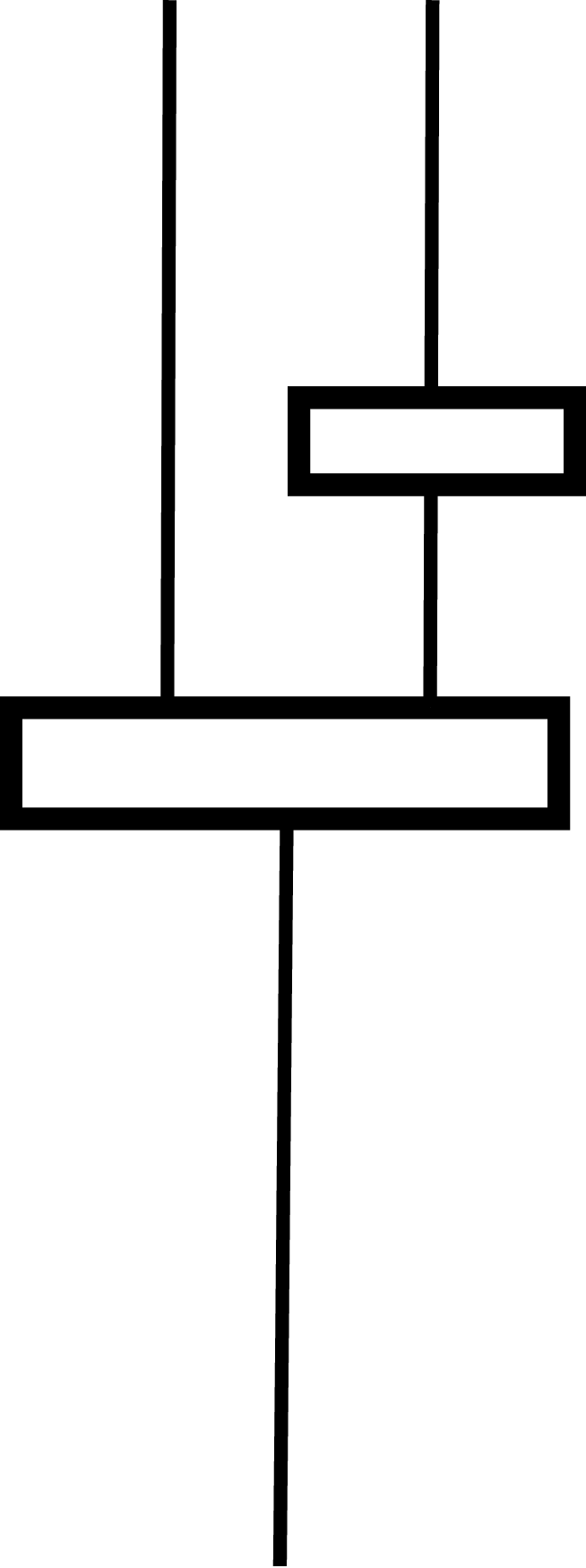}}
        \put(-34,+82){\footnotesize{$n$}}
        \put(-19,+82){\footnotesize{$m$}}
        \put(-46,20){\footnotesize{$m+n$}}
   \end{minipage}
  =
     \begin{minipage}[h]{0.09\linewidth}
        \vspace{0pt}
        \scalebox{0.115}{\includegraphics{idempotent2}}
        \put(-46,20){\footnotesize{$m+n$}}
   \end{minipage}
  \end{eqnarray}
  and
  \begin{eqnarray}
\label{properties2}
    (1)\hspace{8 mm} \begin{minipage}[h]{0.09\linewidth}
        \vspace{0pt}
        \scalebox{0.115}{\includegraphics{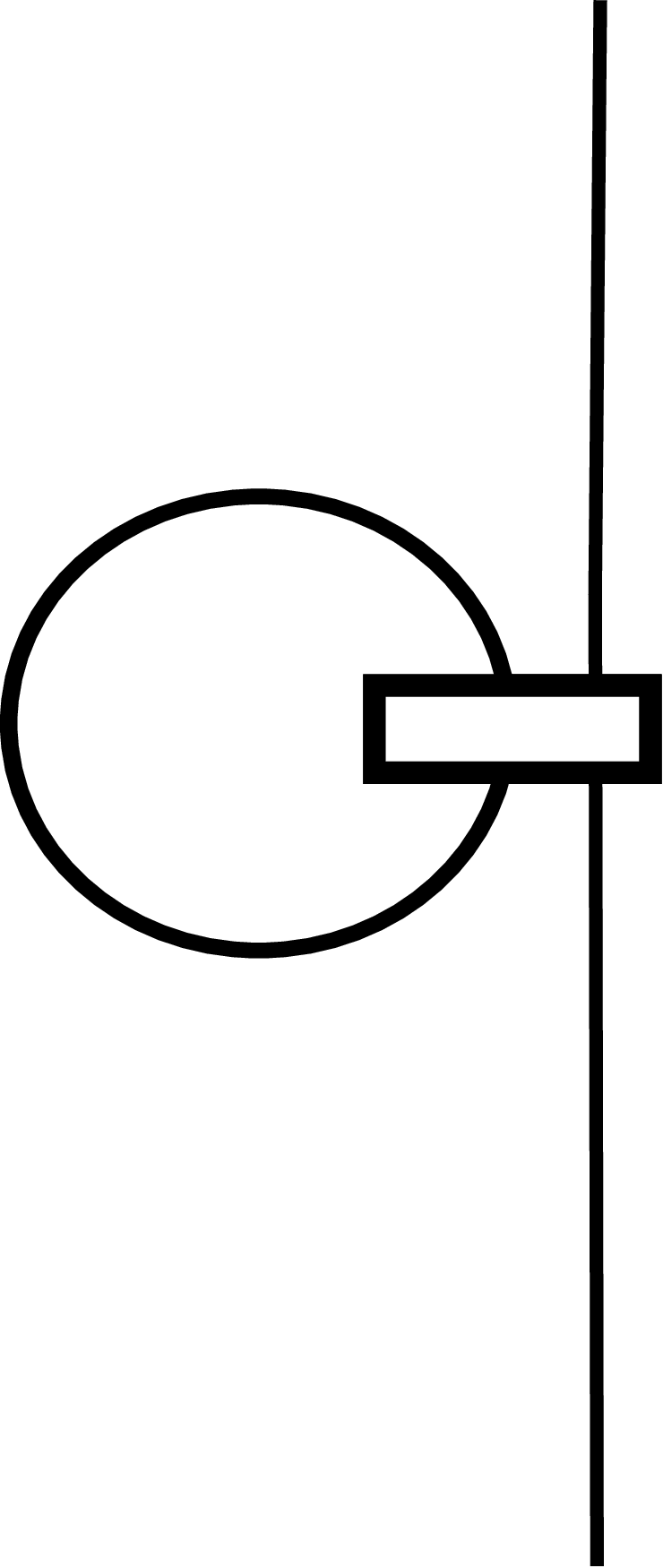}}
          \put(-13,+82){\footnotesize{$n$}}
        \put(-29,+70){\footnotesize{$m$}}
   \end{minipage}
   =\frac{\Delta_{m+n}}{\Delta_{n}}\hspace{1 mm}  
    \begin{minipage}[h]{0.06\linewidth}
        \vspace{0pt}
        \scalebox{0.115}{\includegraphics{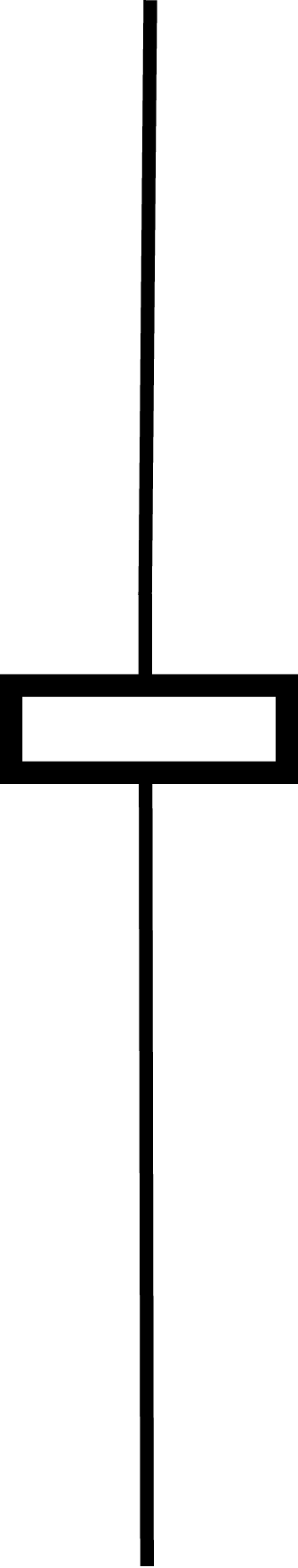}}
           \put(-19,+82){\footnotesize{$n$}}
   \end{minipage}
   , (2) \hspace{8 mm}
   \begin{minipage}[h]{0.08\linewidth}
        \vspace{0 pt}
        \scalebox{.15}{\includegraphics{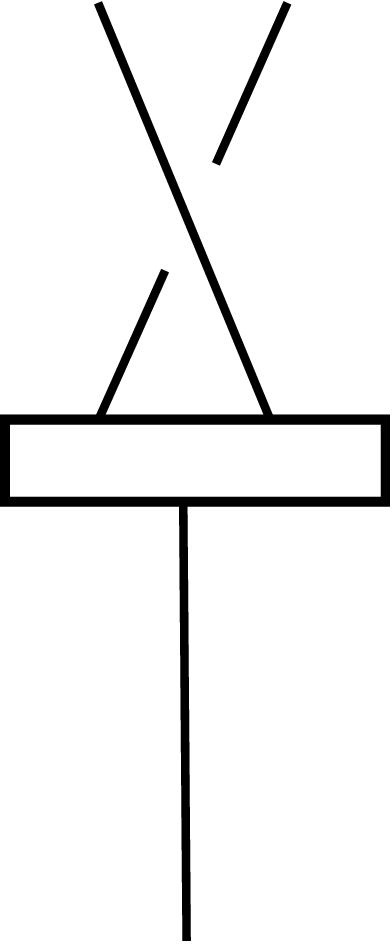}}
        \scriptsize{
        \put(-25,+72){$i$}
        \put(-5,+72){$j$}
        \put(-10,+10){$i+j$}
        }

   \end{minipage}
   =
     A^{-ij}
  \begin{minipage}[h]{0.1\linewidth}
        \vspace{0pt}
        \scalebox{.15}{\includegraphics{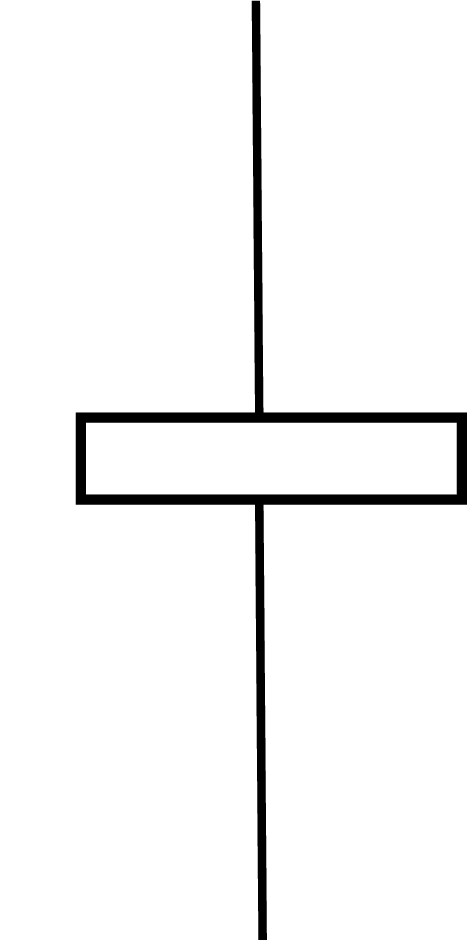}}
        \scriptsize{
          \put(-10,+10){$i+j$}
                  }
          \end{minipage}
   \end{eqnarray}
\subsection{The head and the tail of the colored Jones polynomial for alternating links}  Let $L$ be a zero-framed link in $S^3$. Recall that the unreduced colored Jones polynomial of $L$, denoted by $\tilde{J}_{n,L}(A)$, can be obtained by decorate every component of $L$, according to its framing, by the $n^{th}$ Jones-Wenzl idempotent and consider this decorated framed link as an element of $\mathcal{S}(S^3)$. We are only concerted with the coefficients of the colored Jones polynomial so for our purpose the framing of the link $L$ could be chosen arbitrary.\\

Following \cite{Cody2}, write $P_1(A)\doteq_n P_2(A)$ ,where $P_1(A)$, $P_2(A)$ are two Laurent series if and only if the first $n$ coefficients of $P_1(A)$ and $P_2(A)$ coincide up to a common sign change. The tail of the unreduced colored Jones polynomial of a link $L$, if it exist, is defined to be a series $T_L(A)$, that satisfies  $T_L(A)\doteq_{4n}\tilde{J}_{n,L}(A)$ for all $n$.
\section{The colored Kauffman skein relation}
\label{section3}
We start this section by proving the colored Kauffman skein relation in \ref{lemma1}. This relation is implicit in the work of Yamada in \cite{yamada}. The colored Kauffman skein relation will be used in the next section to understand the highest and the lowest coefficients of the colored Jones polynomial.\\

The following two Lemmas are basically due to Yamada \cite{yamada}. We include the proof here with modification for completeness.
\begin{lemma}
\label{lemma1}(The colored Kauffman skein relation)
Let $n\geq 0$. Then we have
\begin{eqnarray*}
   \begin{minipage}[h]{0.15\linewidth}
        \vspace{0 pt}
        \scalebox{.5}{\includegraphics{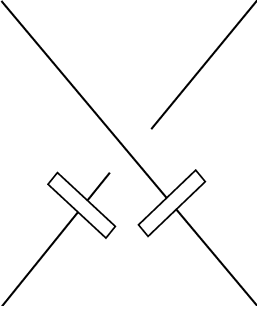}}
        \scriptsize{
         \put(-75,+82){$n+1$}
          \put(-8,+82){$n+1$}}

   \end{minipage}
   =
     A^{2n+1}
  \begin{minipage}[h]{0.20\linewidth}
        \vspace{10pt}
        \scalebox{.5}{\includegraphics{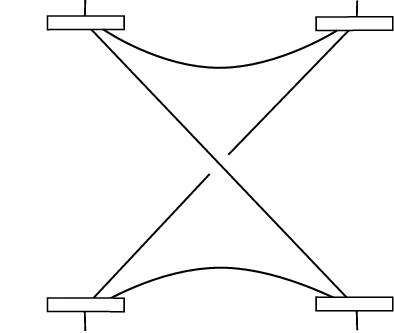}}
        \scriptsize{
        \put(-85,+82){$n+1$}
          \put(-20,+82){$n+1$}
          \put(-60,48){$n$}
          \put(-30,48){$n$}
         \put(-45,+56){$1$}
           \put(-45,+18){$1$}
           }
          \end{minipage}
   +
  A^{-(2n+1)}
  \begin{minipage}[h]{0.20\linewidth}
        \vspace{10pt}
        \scalebox{.5}{\includegraphics{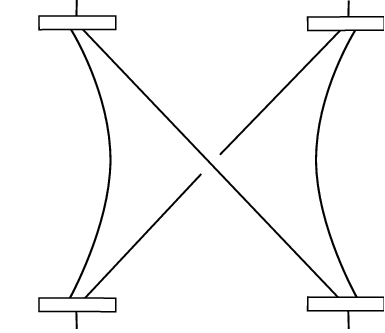}}
        \scriptsize{
        \put(-85,+82){$n+1$}
          \put(-20,+82){$n+1$}
          \put(-60,60){$n$}
          \put(-32,60){$n$}
            \put(-25,+42){$1$}
           \put(-65,+42){$1$}
          
           }
          \end{minipage}
\end{eqnarray*}
\end{lemma}
\begin{proof}
Applying the Kauffman relation we obtain
\begin{eqnarray*}
   \begin{minipage}[h]{0.15\linewidth}
        \vspace{0 pt}
        \scalebox{.5}{\includegraphics{colored_crossing.eps}}
        \scriptsize{
         \put(-75,+78){$n+1$}
          \put(-8,+78){$n+1$}}

   \end{minipage}
   &=&
     A
  \begin{minipage}[h]{0.18\linewidth}
        \vspace{5pt}
        \scalebox{.5}{\includegraphics{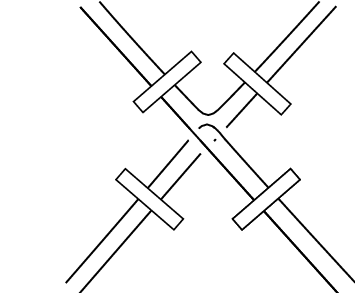}}
        \scriptsize{
        \put(-85,+78){$n+1$}
          \put(-20,+78){$n+1$}
          \put(-24,+64){$1$}
           \put(-50,+64){$1$}
            \put(-24,+10){$n$}
           \put(-50,+10){$n$}
             \put(-1,+10){$1$}
           \put(-70,+10){$1$}
            \put(-9,+57){$n$}
           \put(-65,+57){$n$}
           }
          \end{minipage}
   +
  A^{-1}
  \begin{minipage}[h]{0.18\linewidth}
        \vspace{0pt}
        \scalebox{.5}{\includegraphics{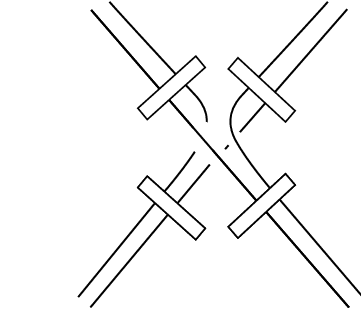}}
        \scriptsize{
        \put(-85,+78){$n+1$}
          \put(-20,+78){$n+1$}
          \put(-24,+64){$1$}
           \put(-50,+64){$1$}
            \put(-24,+10){$n$}
           \put(-50,+10){$n$}
             \put(-1,+10){$1$}
           \put(-70,+10){$1$}
            \put(-9,+57){$n$}
           \put(-65,+57){$n$}
           }
          \end{minipage}\\&=&    A
  \begin{minipage}[h]{0.18\linewidth}
        \vspace{5pt}
        \scalebox{.5}{\includegraphics{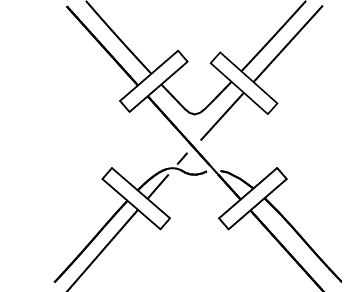}}
        \scriptsize{
        \put(-85,+75){$n+1$}
          \put(-20,+75){$n+1$}
          \put(-24,+64){$1$}
           \put(-50,+64){$1$}
            \put(-24,+10){$n$}
           \put(-50,+10){$n$}
             \put(-1,+10){$1$}
           \put(-70,+10){$1$}
            \put(-9,+57){$n$}
           \put(-65,+57){$n$}
           }
          \end{minipage}
   +
  A^{-1}
  \begin{minipage}[h]{0.18\linewidth}
        \vspace{5pt}
        \scalebox{.5}{\includegraphics{smoothing_A}}
        \scriptsize{
        \put(-85,+75){$n+1$}
          \put(-20,+75){$n+1$}
           \put(-24,+64){$1$}
           \put(-50,+64){$1$}
            \put(-24,+10){$n$}
           \put(-50,+10){$n$}
             \put(-1,+10){$1$}
           \put(-70,+10){$1$}
            \put(-9,+57){$n$}
           \put(-65,+57){$n$}
           }
          \end{minipage}      
\end{eqnarray*}
Using property (2) in \ref{properties2} we obtain the result.
\end{proof}
Note that if we special $n$ to $1$ in the previous Lemma we obtain the Kauffman skein relation.
\begin{lemma}
\label{lemma2}
Let $n\geq 0$. Then we have
\begin{eqnarray*}
   \begin{minipage}[h]{0.15\linewidth}
        \vspace{0 pt}
        \scalebox{.4}{\includegraphics{colored_crossing.eps}}
        \scriptsize{
         \put(-55,+52){$n$}
          \put(-1,+52){$n$}}

   \end{minipage}
   =
     \displaystyle\sum\limits_{k=0}^{n}C_{n,k}
  \begin{minipage}[h]{0.15\linewidth}
        \vspace{0pt}
        \scalebox{.4}{\includegraphics{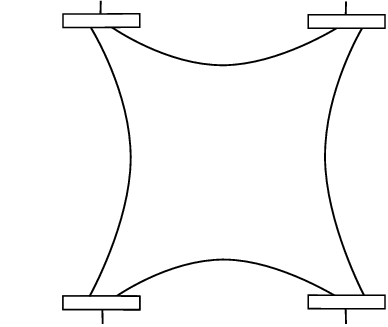}}
        \scriptsize{
        \put(-60,+64){$n$}
          \put(-5,+64){$n$}
          \put(-58,30){$k$}
          \put(-7,30){$k$}
         \put(-43,+56){$n-k$}
           \put(-43,+16){$n-k$}
           }
          \end{minipage}
  \end{eqnarray*}
  Where
  \begin{equation*}
  C_{n,k}=A^{n(n-2k)}{n \brack k}_{A}.
  \end{equation*}
  
\end{lemma}
\begin{proof}
Lemma \ref{lemma1} implies that
\begin{eqnarray}
\label{GOO}
   \begin{minipage}[h]{0.15\linewidth}
        \vspace{0 pt}
        \scalebox{.4}{\includegraphics{colored_crossing.eps}}
        \scriptsize{
         \put(-55,+52){$n$}
          \put(-1,+52){$n$}}

   \end{minipage}
   =
     \displaystyle\sum\limits_{k=0}^{n}C^{\prime}_{n,k}
  \begin{minipage}[h]{0.15\linewidth}
        \vspace{0pt}
        \scalebox{.4}{\includegraphics{goodbasis}}
        \scriptsize{
        \put(-60,+64){$n$}
          \put(-5,+64){$n$}
          \put(-58,30){$k$}
          \put(-7,30){$k$}
         \put(-43,+56){$n-k$}
           \put(-43,+16){$n-k$}
           }
          \end{minipage}
  \end{eqnarray}
where $C^{\prime}_{n,k}$ is a polynomial with integral coefficients in A. Let us prove by induction on $n$ that we have
\begin{equation}
\label{induction}
C^{\prime}_{n,k}=C_{n,k}.
\end{equation}
For $n=1$ relation (\ref{induction}) holds since this is just the Kauffman skein relation. Applying the identity (\ref{GOO}) on the colored each term of the colored Kauffman skein relation, we obtain :
\begin{eqnarray*}
    \displaystyle\sum\limits_{k=0}^{n}C^{\prime}_{n,k}
  \begin{minipage}[h]{0.15\linewidth}
        \vspace{0pt}
        \scalebox{.4}{\includegraphics{goodbasis}}
        \scriptsize{
        \put(-60,+64){$n$}
          \put(-5,+64){$n$}
          \put(-48,30){$k$}
          \put(-7,30){$k$}
         \put(-43,+56){$n-k$}
           \put(-43,+16){$n-k$}
           }
          \end{minipage}
  &=&
   A^{2n-1}\displaystyle\sum\limits_{k=0}^{n-1}C^{\prime}_{n-1,k} \begin{minipage}[h]{0.17\linewidth}
        \vspace{0pt}
        \scalebox{.4}{\includegraphics{goodbasis}}
        \scriptsize{
        \put(-60,+64){$n$}
          \put(-5,+64){$n$}
          \put(-48,30){$k$}
          \put(-7,30){$k$}
         \put(-43,+56){$n-k$}
           \put(-43,+16){$n-k$}
           }
          \end{minipage}+ A^{-2n+1}\displaystyle\sum\limits_{k=0}^{n-1}C^{\prime}_{n-1,k} \begin{minipage}[h]{0.17\linewidth}
        \vspace{0pt}
        \scalebox{.4}{\includegraphics{goodbasis}}
        \scriptsize{
        \put(-60,+64){$n$}
          \put(-5,+64){$n$}
          \put(-48,30){$k+1$}
          \put(-7,30){$k+1$}
           }
          \end{minipage}\\&=&A^{2n-1}\displaystyle\sum\limits_{k=0}^{n-1}C^{\prime}_{n-1,k} \begin{minipage}[h]{0.16\linewidth}
        \vspace{0pt}
        \scalebox{.4}{\includegraphics{goodbasis}}
        \scriptsize{
        \put(-60,+64){$n$}
          \put(-5,+64){$n$}
          \put(-48,30){$k$}
          \put(-7,30){$k$}
         \put(-43,+56){$n-k$}
           \put(-43,+16){$n-k$}
           }
          \end{minipage}+ A^{-2n+1}\displaystyle\sum\limits_{k=1}^{n-1}C^{\prime}_{n-1,k-1} \begin{minipage}[h]{0.15\linewidth}
        \vspace{0pt}
        \scalebox{.4}{\includegraphics{goodbasis}}
        \scriptsize{
        \put(-60,+64){$n$}
          \put(-5,+64){$n$}
          \put(-48,30){$k$}
          \put(-7,30){$k$}
           }
          \end{minipage}
  \end{eqnarray*}
The skein elements  $\begin{minipage}[h]{0.13\linewidth}
        \vspace{0pt}
        \scalebox{.3}{\includegraphics{goodbasis}}
        \scriptsize{
        \put(-50,+48){$n$}
          \put(-5,+48){$n$}
          \put(-45,20){$k$}
          \put(-7,20){$k$}
         \put(-34,+43){$n-k$}
           \put(-34,+14){$n-k$}
           }
          \end{minipage}$, where $0\leq k \leq n$, are linearly independent, see for instance \cite{Hajij}, and hence 
\begin{equation}
\label{rec1}
C^{\prime}_{n,k}=A^{2n-1}C^{\prime}_{n-1,k}+A^{-2n+1}C^{\prime}_{n-1,k-1}.
\end{equation}
However
\begin{eqnarray*}
C_{n,k}&=&A^{n(n-2k)}(A^{2k}{n-1 \brack k}_{A}+A^{2k-2n}{n-1 \brack k-1}_{A})
\\&=&A^{n^2-2nk+2k}{n-1 \brack k}_{A}+A^{n^2-2nk-2n+2k}{n-1 \brack k-1}_{A}
\end{eqnarray*}
Hence
\begin{equation}
\label{rec2}
C_{n,k}=A^{2n-1}C_{n-1,k}+A^{-2n+1}C_{n-1,k-1}.
\end{equation}
Relations (\ref{rec1}) and (\ref{rec2}) and the induction hypothesis yield the result.
\end{proof}
Motivated by Lemma \ref{lemma1} we define \textit{the $n$-colored states} for a link diagram $D$ for every positive integer $n$. Suppose that the link diagram has $k$ crossings. Label the crossings of the link diagram $D$ by $1$,..,$k$. An $n$-colored state $s^{(n)}$ for a link diagram $D$ is a function $s^{(n)}:\{1,..,k\}\rightarrow\{-1,+1\}$. If the color $n$ is clear from the context, we will drop $n$ from the notation of a colored state.
\begin{figure}[H]
  \centering
   {\includegraphics[scale=0.5]{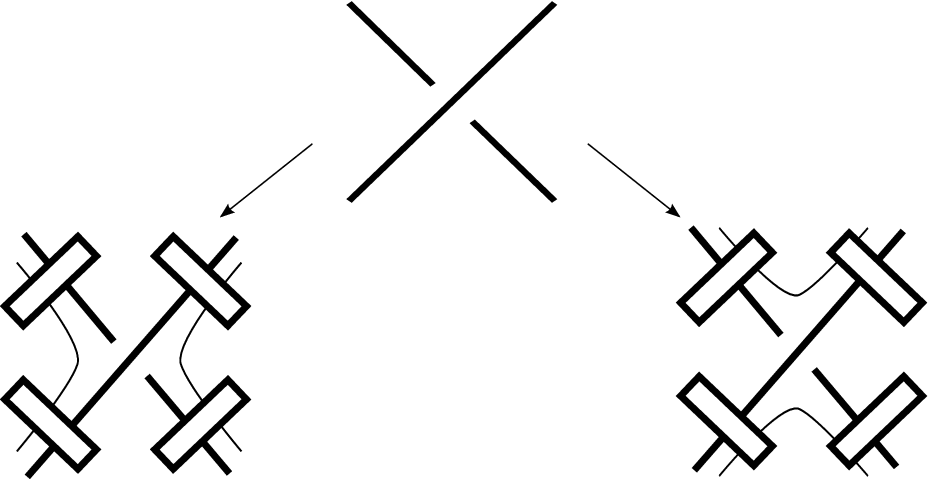}
   \small{
    \put(-165,78){$A$}
          \put(-70,78){$B$}
          \put(-70,26){$n-1$}
          \put(-13,26){$n-1$}
          \put(-35,52){$1$}
           \put(-35,6){$1$}
 		\put(-215,25){$1$}
          \put(-178,25){$1$}
          \put(-203,63){$n-1$}
           \put(-203,-7){$n-1$}           
           }
    \caption{The $n$-colored $A$ and $B$ smoothings}
  \label{smoothings}}
\end{figure}
Given a link diagram $D$ and a colored state $s$ for the diagram $D$. We construct a skein element $\Upsilon^{(n)}_D(s)$ obtained from $D$ by replacing each crossing labeled $+1$ by an $n$-colored $A$-smoothing and each crossing labeled $-1$ by an $n$-colored $B$-smoothing, see Figure \ref{smoothings}. Two particular skein elements obtained in this way are important to us. The skein element obtained by replacing each crossing by the $n$-colored $A$-smoothing will be called the \textit{$n$-colored $A$-state} and denoted by $\Upsilon^{(n)}_D(s_+)$, where $s_+$ denotes the colored state for which $s_+(i)=+1$ for all $i$ in $\{1,..,m\}$. The $n$-colored $B$-state is defined similarly. See Figure \ref{colored state} for an example.
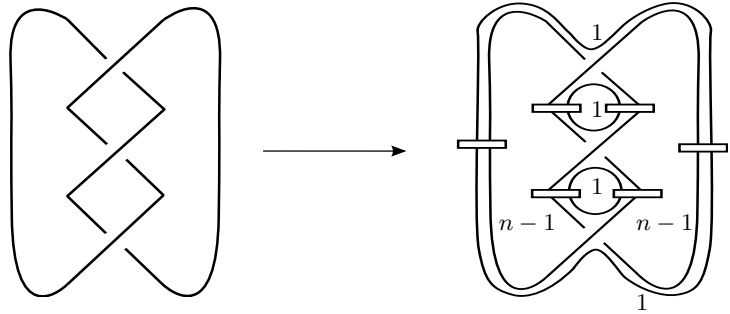
\begin{figure}[H]
  \centering
   {\includegraphics[scale=0.6]{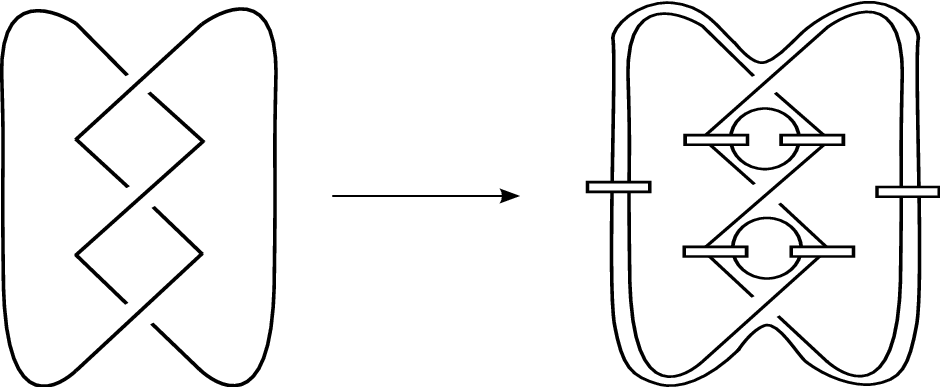}
    \footnotesize{
    \put(-87,26){$n-1$}
          \put(-35,26){$n-1$}
          \put(-52,40){$1$}
          \put(-52,69){$1$}
          \put(-52,98){$1$}
           \put(-35,-4){$1$}
           }
       \caption{$n$-colored $B$-state}
  \label{colored state}}
\end{figure}
Following Armond \cite{Cody2}, we now define what we mean by an adequate skein element in $\mathcal{S}(S^2)$. Consider a skein element $S$ in $\mathcal{S}(S^2)$ consists of arcs connecting Jones-Wenzl idempotents of various colors. Construct the diagram $\bar{S}$ from $S$ by replacing each $i^{th}$ Jones-Wenzl idempotent in $S$ by the identity in $TL_i$. If $S$ contains no crossings then we say that $S$ is adequate if $\bar{S}$ consists of circles each one of them passes at most once from the regions where we had the boxes of the idempotents in $S$. In other words, a crossingless skein element $S$ is adequate if each circle in $S$ passes at most once from the same idempotent. See Figure \ref{cody_1} for a local picture of an adequate skein element.
 \begin{figure}[H]
  \centering
   {\includegraphics[scale=0.1]{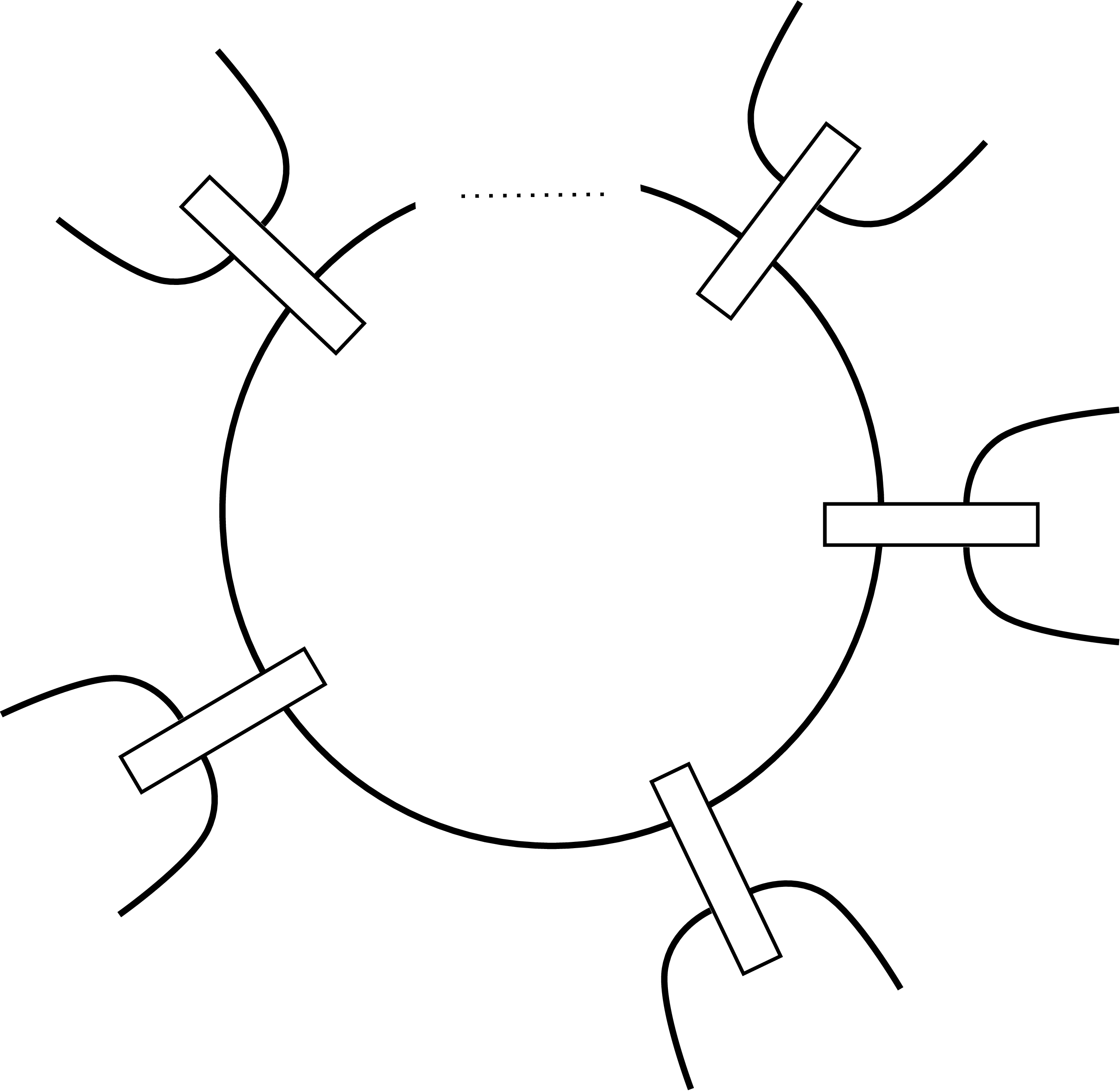}
    \put(-60,80){$n$}
          \put(-15,79){$k$}
    \caption{A local picture an adequate skein element}
    \label{cody_1}
 }
\end{figure} 
If $f$ is in $\mathcal{S}(S^2)$ then $d(f)$ will denote the minimum degree of $f$ expressed as a
Laurent series in $A$. Furthermore, denote $D(S):=d(\bar{S})$. The following lemma is due to Armond \cite{Cody2}.
\begin{lemma}
\label{codys}
If $S \in \mathcal{S}(S^2)$ is expressed as a single diagram containing the Jones-Wenzl idempotent, then $d(S)\geq D(S)$. If the diagram for $S$ is an adequate skein diagram, then $d(S) = D(S)$.
\end{lemma}
\section{The main theorem}
The colored Kauffman skein relation provides a natural framework to understand the highest and the lowest coefficients of the colored Jones polynomial. In this section we will use this relation to prove that the the highest (the lowest respectively) $4n$ coefficients of the $n^{th}$ unreduced colored Jones polynomial agree up to a sign with the $n$-colored $A$-state (the $n$-colored $B$-state respectively). We use this result to prove existence of the the tail of the colored Jones polynomial.
\begin{theorem}
\label{thm1}
 Let $L$ be an alternating link diagram. Then
\begin{equation*}
\tilde{J}_{n,L}\doteq_{4n}\Upsilon_L^{(n)}(s_-)
\end{equation*}
\end{theorem}
\begin{proof}
Assume that the link diagram $L$ has $k$ crossings and label the crossing of the link diagram by $1,..,k$. The colored Kauffman skein relation implies that 
\begin{equation*}
\tilde{J}_{n,L}\doteq\sum_{s} \alpha_L(n,s)\Upsilon_L^{(n)}(s)
\end{equation*}
where  $\alpha_L(n,s)=A^{(2n-1)\sum_{i=1}^{k}s(i)}$ and the summation runs over all functions $s:\{1,2,...,k\}\rightarrow\{-1,+1\}$. Now for any colored state $s$ of the link diagram $L$ there is a sequence of states $s_0,s_1...,s_r$ such that $s_0=s_-$ , $s_r=s$ and $s_{j-1}(i)=s_j(i)$ for all $i \in\{1,...,k\}$ except for one integer $i_l$ for which $s_{j-1}(i_l)=-1$ and $s_{j}(i_l)=1$. It is enough to show that the lowest $4n$ terms of $\alpha(n,s_-)\Upsilon_L^{(n)}(s_-)$ are never canceled by any term from  $\alpha(n,s)_L\Upsilon_L^{(n)}(s)$ for any $s$. For any colored state $s$ of the link diagram $L$ one could write 
\begin{equation*}
\Upsilon_L^{(n)}(s)=\displaystyle\sum\limits_{i_1,...,i_k=0}^{n-1}\prod_{j=1}^{k} C_{n-1,i_j}\Lambda_{s,(i_1,...,i_k)}
\end{equation*}
where $\Lambda_{s,(i_1,...,i_k)}$ is the skein element that we obtain by applying \ref{lemma2} to every crossing in $\Upsilon_L^{(n)}(s)$. The theorem follows from the following three lemmas.

\end{proof}
\begin{lemma}
\label{m7}
\begin{equation*}
d(\alpha_L(n,s_-))=d(\alpha_L(n,s_1))-4n+2
\end{equation*}
\begin{equation*}
d(\alpha_L(n,s_r))\leq d(\alpha_L(n,s_{r+1}))
\end{equation*}

$$d(C_{n-1,n-1})-d(C_{n-1,n-2})=-2$$  $$d(C_{n-1,i})\leq d(C_{n-1,i-1}).$$
\end{lemma}
\begin{proof}It is clear that
\begin{equation*}
\alpha_L(n,s_-)=A^{k - 2 k n}
\end{equation*}
and
\begin{equation*}
\alpha_L(n,s_1)=A^{-2 + k + 4 n - 2 k n}
\end{equation*}
Furthermore,
\begin{equation*}
\alpha_L(n,s_r)=A^{(2n-1)\sum_{i=1}^{k}s_r(i)}=A^{(2n-1)(-k+2r)}=A^{k - 2 k n - 2 r + 4 n r}
\end{equation*}
hence 
\begin{equation*}
d(\alpha_L(n,s_r))-d(\alpha_L(n,s_{r+1}))=2-4n.
\end{equation*}
Finally, it is clear that $$d(C_{n,i})=2i^2-4in+n^2.$$
Hence, the result follows.
\end{proof}

\begin{lemma}
\begin{equation*}
d(\Lambda_{s_-,(n-1,...,n-1)})= D(\Lambda_{s_1,(n-1,...,n-1)})-2.
\end{equation*}

\begin{equation*}
D(\Lambda_{s,(i_1,...i_{j-1},i_j,i_{j+1}...,i_k)})= D(\Lambda_{s,(i_1,...i_{j-1},i_j-1,i_{j+1},...,i_k)})\pm2.
\end{equation*}
\end{lemma}
\begin{proof}
When we replace the idempotent by the identity in the skein element $\Upsilon^{(n)}(s_-)$ we obtain the diagram $L^{n-1}\dot{\bigcup}C$ where $C$ is a link diagram composed of a disjoint union of unit circles each one of them bounds a disk and $L^{n-1}$ is the $(n-1)$-parallel of $L$. Note that state $\overline{\Lambda_{s_-,(n-1,...,n-1)}}$ is exactly the all $B$-state of the link diagram $L^n$ and it is also the all B-state of $L^{n-1}\dot{\bigcup}C$. Since $L$ is alternating then the the number of circles in $\overline{\Lambda_{s_1,(n-1,...,n-1)}}$ is one less than the number of circles in $\overline{\Lambda_{s_-,(n-1,...,n-1)}}$. In other words, the number of circles in the all $B$-state of $\overline{\Upsilon^{(n)}(s_1)}$ is one less than the number of circles in the all $B$-state of the diagram $L^{n-1}\dot{\bigcup}C$. Thus,
$$D(\Lambda_{s_-,(n-1,...,n-1)})= D(\Lambda_{s_1,(n-1,...,n-1)})-2.$$
 Moreover the skein element $\Lambda_{s_-,(n-1,...,n-1)}$ is adequate since $L$ is alternating. Hence by Lemma \ref{codys} we have $$d(\Lambda_{s_-,(n-1,...,n-1)})= D(\Lambda_{s_1,(n-1,...,n-1)})-2.$$ For the second part, note that the number of circles in the diagrams $\overline{\Lambda_{s,(i_1,...i_{j-1},i_j,i_{j+1}...,i_k)}}$ and $\overline{\Lambda_{s,(i_1,...i_{j-1},i_j-1,i_{j+1},...,i_k)}}$  differs by $1$. Hence by \ref{codys} we obtain $$D(\Lambda_{s,(i_1,...i_{j-1},i_j,i_{j+1}...,i_k)})= D(\Lambda_{s,(i_1,...i_{j-1},i_j-1,i_{j+1},...,i_k)})\pm2$$. 
\end{proof}
\begin{lemma}
\begin{equation*}
d(\Upsilon^{(n)}(s_-))=d((C_{n-1,n-1})^k\Lambda_{s_-,(n-1,...,n-1)}).
\end{equation*}
\end{lemma}
\begin{proof}
The previous two lemmas imply directly that the lowest term in $\Upsilon^{(n)}(s_-)$ is coming from the skein element $(C_{n-1,n-1})^k\Lambda_{s_-,(n-1,...,n-1)}$ and this term is never canceled by any other term in the summation.
\end{proof}

\begin{theorem}
\label{thm2}
Let $L$ be an alternating link diagram and let $\Upsilon^{(n+1)}(s_-)$ be its corresponding $(n+1)$-colored $B$-state skein element. Then
\begin{equation*}
\Upsilon_L^{(n+1)}(s_-)\doteq_{4n}\tilde{J}_{n,L}
\end{equation*}
\end{theorem}
\begin{proof}
Since $L$ is an alternating link diagram then the skein element $\Upsilon^{(n+1)}(s_-)$ must look locally as in Figure \ref{local}.
\begin{figure}[H]
  \centering
   {\includegraphics[scale=0.25]{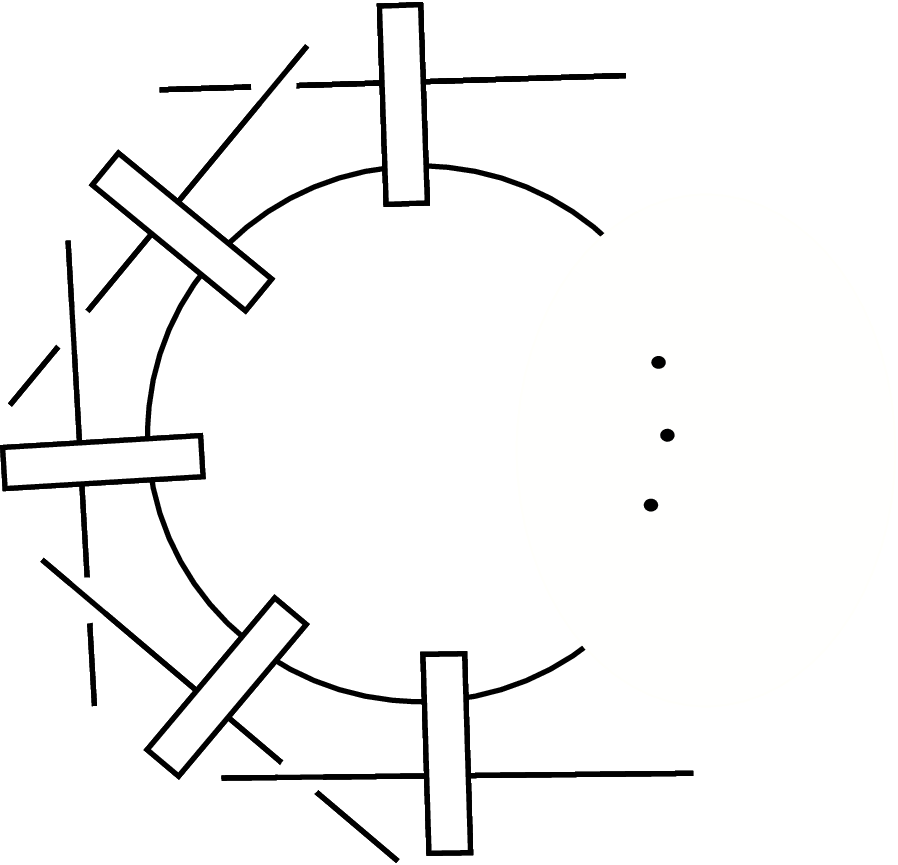}
    \put(-85,55){$1$}
    \put(-32,2){$n$}
     \caption{Local view of the skein element $\Upsilon^{(n+1)}(s_-)$}
  \label{local}}
\end{figure}
It follows from Theorem $9$ and Lemma $10$  in \cite{Cody2} that
\begin{eqnarray*}
\label{properties}
\hspace{0 mm}
    \begin{minipage}[h]{0.21\linewidth}
        \vspace{0pt}
        \scalebox{0.24}{\includegraphics{proof_1}}
         \put(-80,55){$1$}
    \put(-32,2){$n$}
               
   \end{minipage}
  \doteq_{4n} \hspace{5pt}
     \begin{minipage}[h]{0.1\linewidth}
        \vspace{0pt}
         \hspace{50pt}
        \scalebox{0.24}{\includegraphics{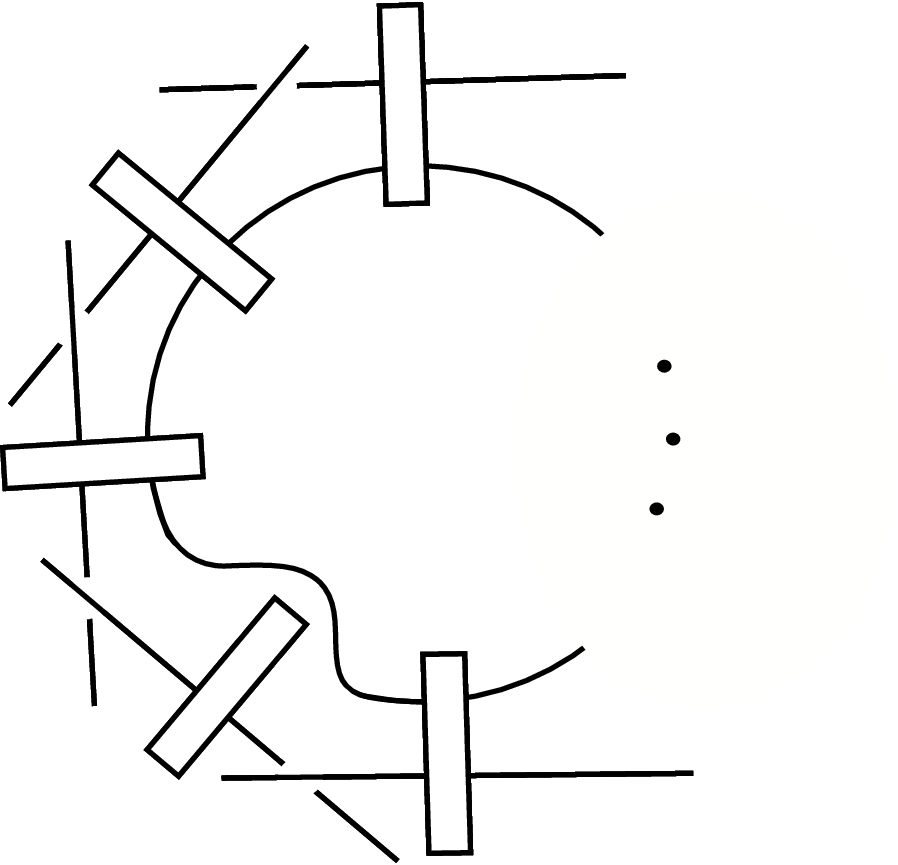}}
          \put(-80,55){$1$}
    \put(-32,2){$n$}
       \end{minipage}
    \end{eqnarray*}
The details of the previous equation can be found in \cite{Cody2} and we will not repeat them here.

 The previous step can be applied around the circle until we reach the final idempotent :
    \begin{eqnarray*}
\hspace{0 mm}
    \begin{minipage}[h]{0.18\linewidth}
        \vspace{0pt}
        \scalebox{0.24}{\includegraphics{proof_1}}
         \put(-80,55){$1$}
    \put(-32,2){$n$}
               
   \end{minipage}
  \doteq_{4n} \hspace{5pt}
     \begin{minipage}[h]{0.1\linewidth}
        \vspace{0pt}
         \hspace{50pt}
        \scalebox{0.24}{\includegraphics{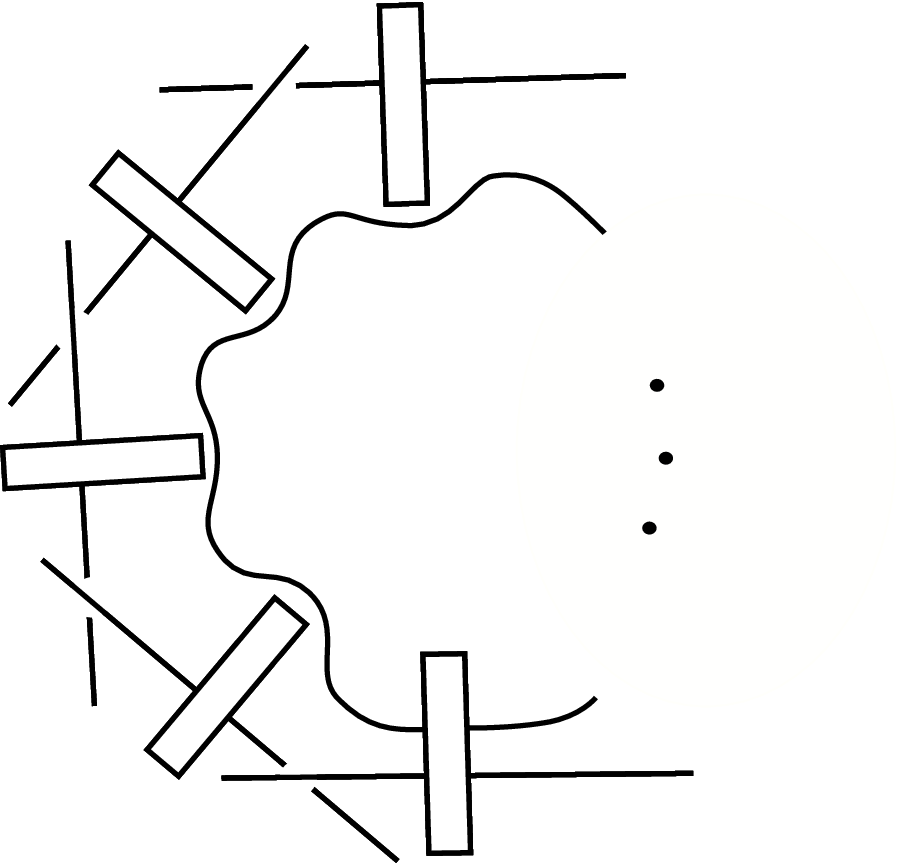}}
          \put(-70,55){$1$}
    \put(-32,2){$n$}
       \end{minipage}\hspace{40pt}\doteq_{4n} 
     \begin{minipage}[h]{0.1\linewidth}
        \vspace{0pt}
         \hspace{200pt}
        \scalebox{0.24}{\includegraphics{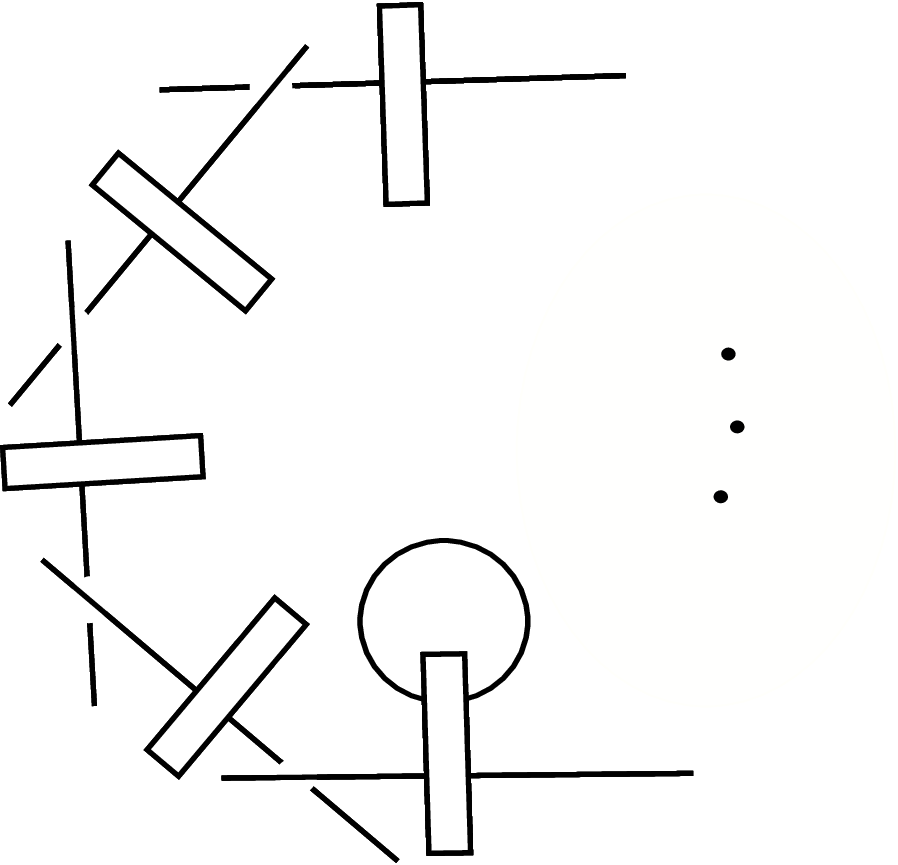}}
          \put(-60,38){$1$}
    \put(-32,2){$n$}
       \end{minipage}
    \end{eqnarray*}
Equation \ref{properties2} implies   
 \begin{eqnarray*}
     \begin{minipage}[h]{0.1\linewidth}
        \vspace{0pt}
         \hspace{50pt}
        \scalebox{0.24}{\includegraphics{proof_4}}
         \put(-60,38){$1$}
    \put(-32,2){$n$}
       \end{minipage}\hspace{40pt}\doteq_{4n}\frac{\Delta_{n+1}}{\Delta_{n}} \hspace{5pt}
     \begin{minipage}[h]{0.13\linewidth}
        \vspace{0pt}
         \hspace{200pt}
        \scalebox{0.24}{\includegraphics{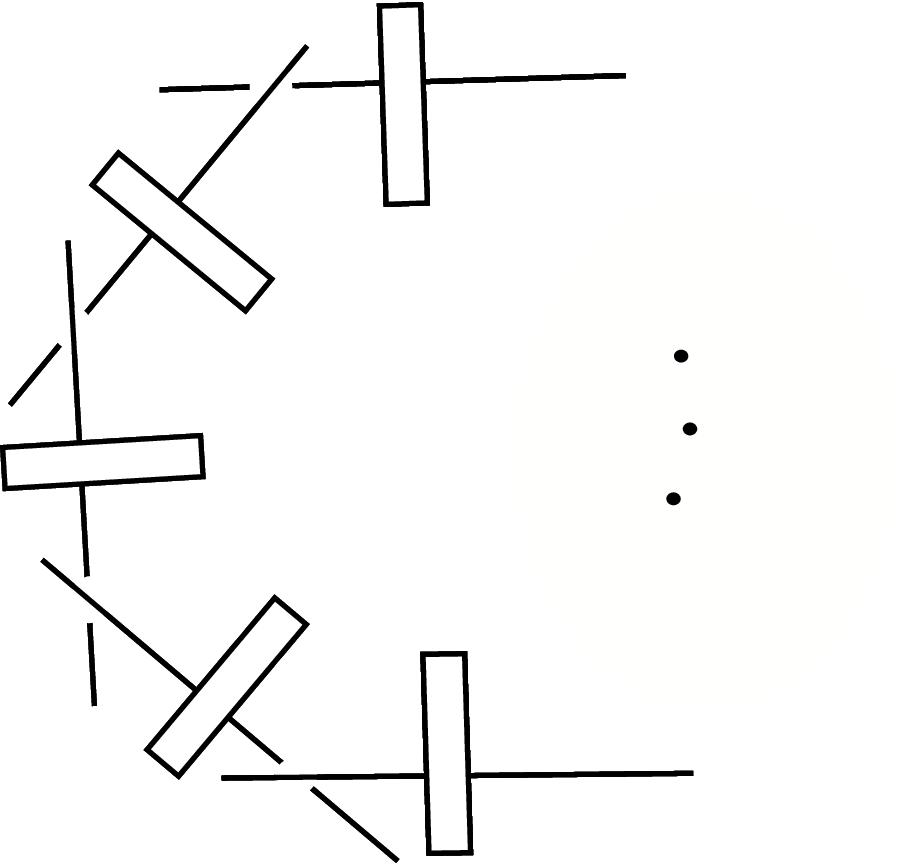}}
        \put(-32,2){$n$}
       \end{minipage}\hspace{40pt}\doteq_{4n}
     \begin{minipage}[h]{0.11\linewidth}
        \vspace{0pt}
         \hspace{205pt}
        \scalebox{0.24}{\includegraphics{proof_5}}
        \put(-32,2){$n$}
       \end{minipage}
    \end{eqnarray*}
Applying this procedure on every circle in $\Upsilon^{(n+1)}(s_-)$, we eventually obtain
\begin{equation*}
\Upsilon^{(n+1)}(s_-)\doteq_{4n}\tilde{J}_{n,L}.
\end{equation*}

\end{proof}
Theorems \ref{thm1} and \ref{thm2} imply immediately the following result.
\begin{corollary}Let $L$ be an alternating link diagram. Then
\begin{equation*}
\tilde{J}_{n+1,L}\doteq_{4n}\tilde{J}_{n,L}
\end{equation*}

\end{corollary}

\section*{Acknowledgements}
I am grateful for Oliver Dasbach for his guidance and advice. I am very thankful
to Kyle Istvan for carefully reading the paper and offering many useful remarks. I also would like to thank Cody Armond for various conversation about this paper.

\end{document}